\documentclass{amsart}

\usepackage{pst-plot}
\usepackage{fullpage, graphicx} 
\usepackage{hyperref} 
\usepackage{amsmath, amssymb, amsthm}
\usepackage{xcolor} 
\usepackage{enumitem}
\usepackage{thmtools}
\usepackage{multirow}
\usepackage{booktabs}

\numberwithin{equation}{subsection}

\hypersetup{colorlinks=true, urlcolor=purple, citecolor=blue, linkcolor=red, pdfstartview={XYZ null null 0.90}}

\DeclareMathOperator{\SL}{SL}
\DeclareMathOperator{\GL}{GL}

\newcommand{\Mod}[1]{\ (\mathrm{mod}\ #1)}

\newcommand{\RR}{\mathbb{R}}
\newcommand{\ZZ}{\mathbb{Z}}
\newcommand{\PP}{\mathbb{P}}
\newcommand{\kron}[2]{\left(\frac{#1}{#2}\right)}
\newcommand{\sm}[4]{\ensuremath{\left(\begin{smallmatrix} #1 & #2\\#3 & #4\end{smallmatrix}\right)}}
\newcommand{\lm}[4]{\ensuremath{\left(\begin{matrix} #1 & #2\\#3 & #4\end{matrix}\right)}}
\newcommand{\genmtx}{\sm{a}{b}{c}{d}}

\newcommand{\smcol}[2]{\ensuremath{\left(\begin{smallmatrix}#1\\#2\end{smallmatrix}\right)}}
\newcommand{\lmcol}[2]{\ensuremath{\left(\begin{matrix}#1\\#2\end{matrix}\right)}}
\newcommand{\opart}[1]{\ensuremath{#1^{\text{o}}}}

\newtheorem{theorem}{Theorem}[section]
\newtheorem{conjecture}[theorem]{Conjecture}
\newtheorem{corollary}[theorem]{Corollary}
\newtheorem{lemma}[theorem]{Lemma}
\newtheorem{proposition}[theorem]{Proposition}
\newtheorem*{zaremba}{Zaremba's Conjecture}
\newtheorem*{kontorovichLG}{Local-Global Conjecture for Continued Fraction Alphabets}

\theoremstyle{definition}

\newtheorem{definition}[theorem]{Definition}

\newtheorem{remark}[theorem]{Remark}

\newtheorem{question}[theorem]{Question}

\begin{document}

\title{Reciprocity obstructions in semigroup orbits in $\SL(2, \ZZ)$}

\author{James Rickards}
\address{Saint Mary's University, Halifax, Nova Scotia, Canada}
\email{james.rickards@smu.ca}
\urladdr{https://jamesrickards-canada.github.io/}

\author{Katherine E. Stange}
\address{University of Colorado Boulder, Boulder, Colorado, USA}
\email{kstange@math.colorado.edu}
\urladdr{https://math.katestange.net/}

\date{\today}

\thanks{Rickards and Stange were supported by NSF-CAREER CNS-1652238 (PI Katherine E. Stange); Stange was also supported by NSF DMS-2401580.  This work utilized the Alpine high performance computing resource at the University of Colorado Boulder. Alpine is jointly funded by the University of Colorado Boulder, the University of Colorado Anschutz, and Colorado State University.}

\keywords{Thin semigroup, thin group, reciprocity obstruction, continued fraction, Zaremba's conjecture, local-global conjecture, Hausdorff dimension}

\subjclass{Primary:  11A15, 14G12, 11J70.  Secondary:  11A55.}

\begin{abstract}
	We study orbits of semigroups of $\text{SL}(2,\mathbb{Z})$, and demonstrate \emph{reciprocity obstructions}:  we show that certain such orbits avoid squares, but not as a consequence of obstructions inherited from an algebraic set, and not as a consequence of congruence obstructions.  This is in analogy to the reciprocity obstructions recently used to disprove the Apollonian local-global conjecture.  We give an example of such an orbit which is known exactly, and misses all squares together with an explicit finite list of sporadic values: the corresponding semigroup is not thin, but is dense in an algebraic variety that does not have such obstructions.  We also demonstrate thin semigroups with reciprocity obstructions, including semigroups associated to continued fractions formed from finite alphabets.  Zaremba's conjecture states that for continued fractions with coefficients chosen from $\{1,\ldots,5\}$, every positive integer appears as a denominator.  Bourgain and Kontorovich proposed a generalization of Zaremba's conjecture in the context of semigroups associated to finite alphabets.  We disprove their conjecture.  In particular, we demonstrate classes of finite continued fraction expansions which never represent rationals with square denominator, but not as a consequence of congruence obstructions, and for which the limit set has Hausdorff dimension exceeding $1/2$.  An example of such a class is continued fractions of the form $[0; a_1, a_2, \ldots, a_n,1,1,2]$, where the $a_i$ are chosen from the set $\{4,8,12,\ldots,128\}$.   The object at the heart of these results is a semigroup $\Psi\subseteq\Gamma_1(4)$ which preserves Kronecker symbols.
\end{abstract}

\maketitle

\setcounter{tocdepth}{1}
\tableofcontents

\section{Introduction}

Let $\Gamma\subseteq\SL(d, \ZZ)$ be a multiplicative semigroup, $v\in\ZZ^d$ a primitive vector, and $L:\ZZ^d\rightarrow\ZZ$ a primitive linear functional. We call $\Gamma v$ a semigroup orbit and study the set $L(\Gamma v) \subseteq \ZZ$ of integers it represents.  Of particular interest are \emph{thin orbits}, which grow slowly, arising from semigroups which are too `small' to inherit their behaviour from their Zariski closure.  In analogy to the study of integer and rational points in Diophantine geometry, we study the arithmetic of these orbits.  
Observed obstructions to integers $n$ appearing in $L(\Gamma v)$ for $\Gamma$ a thin semigroup could be classified as follows:
\begin{enumerate}
    \item \textbf{Definiteness:} the orbit may only represent integers above or below a certain cutoff.  This occurs for Apollonian circle packings; see also \cite[Section 1.2]{BourgainKontorovichGAFA2010} for an example in $\SL(2,\ZZ)$.
    \item \textbf{Counting:} the number of vectors in an orbit of size up to $N$ is typically asymptotic to $cN^{\nu}$ for some positive constants $c, \nu$; for $\nu < 1$ it is not possible to represent all integers.  See \cite{ApolloniusZaremba} for many examples.
    \item \textbf{Congruence:} numbers from a congruence class $a\pmod{m}$ may not be represented by the orbit.  In the case of Apollonian circle packings, $m=24$ describes such obstructions.  See \cite[pp. 2-3]{BourgainKontorovichZarembaAnnals} for an example in $\GL(2,\ZZ)$.
    \item \textbf{Inherited:}  we include in this category any obstruction which exists for the integral or rational points of an algebraic set containing the orbit.  For example, the Zariski closure may be an algebraic group with obstructions inherited by its subgroups and subsemigroups; an example with Brauer-Manin obstruction is mentioned in \cite[p. 5]{KontorovichZhang}.  Variations on this theme include integral Brauer-Manin obstructions such as spinor exceptions for quadratic forms \cite{SP}.  In this paper, the Zariski closure is $\SL(2)$, and our orbits lie in the congruence subgroup $\Gamma_1(4)$, but our obstructions do not; see the discussion following Theorem~\ref{thm:TableForPsi}.
\end{enumerate}
Call an integer \emph{admissible} if it satisfies the definiteness condition and is not ruled out by any congruence condition for an orbit. If $\nu>1$, and there are no further inherited obstructions, then one might expect that every large enough admissible integer should appear.  This is called the \emph{local-global principle}.

This has been conjectured in many cases, including by Zaremba for continued fractions built out of certain finite alphabets \cite[p. 76]{Zaremba72}, later refined by Niederreiter in \cite{Niederreiter78}, and by Hensley in \cite{Hensley96}.  The conjecture is most commonly stated as follows. 

\begin{zaremba}
There exists $n > 0$ so that every positive integer is the denominator of a rational number with continued fraction expansion whose coefficients are at most $n$.
\end{zaremba}

Zaremba also conjectured that $n=5$ is sufficient, and Hensley gave evidence towards $n=2$ being sufficient for all sufficiently large integers. More generally, one might consider any finite alphabet $\mathcal{A}$ for the allowable coefficients, and ask whether the set of denominators contains all but finitely many positive integers.  This statement can be phrased as a local-global principle for the semigroup $\Gamma_\mathcal{A} := \langle \sm{0}{1}{1}{a} : a \in \mathcal{A} \rangle^+ \cap \SL(2,\ZZ)$ corresponding to the alphabet $\mathcal{A}$.  Write $\delta_\mathcal{A}$ for the Hausdorff dimension of the corresponding limit set.  Bourgain and Kontorovich conjectured the following in various forms\footnote{Kontorovich assures us the level of generality given here is that which was intended.}.

\begin{kontorovichLG}[{\cite[Conjecture 6.3.1]{KontorovichThin}, \cite[Conjecture 1.11]{BourgainKontorovichIV}]}]
    Let $\mathcal{A} \subseteq \mathbb{N}$, and consider a linear functional $L$. 
 Suppose $\Gamma_\mathcal{A}$ is Zariski dense in $\SL(2)$, has Hausdorff dimension $\delta_\mathcal{A} > 1/2$, and that $L(\Gamma_\mathcal{A} v)$ is infinite.  In terms of a growing parameter $X$, any admissible integer $n \asymp X$ has multiplicity $X^{2\delta_\mathcal{A}-1-o(1)}$ in $L(\Gamma_\mathcal{A} v)$.
\end{kontorovichLG}

For Apollonian circle packings, the local-global conjecture is due to Graham-Lagarias-Mallows-Wilks-Yan and Fuchs-Sanden in \cite{GLMWY02}, \cite{FS11}. Variants of the conjecture for generalizations of circle packings have been given by Zhang \cite{Zhang18} and jointly by Fuchs, Zhang, and the second author \cite{FSZ19}. 

In most cases, the best known result is that a density one of admissible integers appear.  For Zaremba's conjecture, this is due to Bourgain-Kontorovich \cite{BourgainKontorovichZarembaAnnals} and later Frolenkov and Kan \cite{FrolenkovKan}, Huang \cite{Huang15} and Kan \cite{Kan}; for Apollonian circle packings this is due to Bourgain-Kontorovich
\cite{BK14} and later extended by Zhang \cite{Zhang18} and jointly by Fuchs, Zhang, and the second author \cite{FSZ19} for certain families of Kleinian circle packings.  The only proof of the full conjecture is by Kontorovich for Soddy sphere packings \cite{KontorovichSoddy}.  For an introduction, see \cite{ApolloniusZaremba}.

Recently, the local-global conjecture was proven false for many Apollonian circle packings by the authors, in collaboration with Haag and Kertzer \cite{HKRS23}, by identifying \emph{reciprocity obstructions}.  These describe infinite families of admissible integers that do not appear.  The culprit is a reciprocity law, specifically quadratic and quartic reciprocity.

We expect these reciprocity obstructions to exist for many other (thin) (semi)groups. In this paper, we study subsemigroups of $\SL(2, \ZZ)$ and exhibit reciprocity obstructions.  We consider more specifically semigroups whose matrix entries are non-negative, to which we can associate limit sets consisting of real numbers with continued fraction expansions satisfying certain restrictions.  One might refer to these as \emph{continued fraction semigroups}.  

In particular, we disprove the Local-Global Conjecture for Continued Fraction Alphabets.   Precisely, we exhibit reciprocity obstructions for the linear functional $L(x,y) = 3x + 5y$ and the alphabet $\mathcal{A} = \{4,8,12,\ldots,128\}$.  An equivalent way to state this is that rational numbers with continued fractions of the form $[0; a_1, a_2, \ldots, a_n,1,1,2]$, where the $a_i \in \mathcal{A}$, cannot have square denominators, despite squares not being ruled out by congruence, and despite the average multiplicity of denominators tending to infinity.  We do not disprove Zaremba's conjecture itself.

In our context, the notion of \emph{thin} typically means that the semigroup itself or its orbits grow more slowly than those of the Zariski closure.  A precise relationship between orbit growth and Hausdorff dimension is due to Hensley in the case of a semigroup $\Gamma_\mathcal{A} \subseteq\GL(2, \ZZ)$ associated to the collection of continued fractions whose coefficients are chosen from a finite alphabet $\mathcal{A}$.  In this case, Hensley \cite{Hensley90} shows that
\begin{equation}
    \label{eqn:hensley}
\# \{ w \in \Gamma_\mathcal{A} v : ||w||_\infty < N \} \sim C_{\mathcal{A}} N^{2\delta_\mathcal{A}},
\end{equation}
where $\delta_\mathcal{A}$ is the Hausdorff dimension of the associated limit set.  The existence of a local-global principle is a sensible question for such semigroups with Hausdorff dimension exceeding $1/2$, since in this case the average multiplicity of a represented integer tends to $\infty$ as $N\rightarrow\infty$. Although Hensley's result is not proven for all continued fraction semigroups, we choose to state our results in terms of Hausdorff dimension.

A very brief summary of results is contained in the abstract, and the following section contains more detailed statements.

\subsection*{Software}  Computational methods were developed and implemented in C and PARI/GP \cite{PARI} to investigate orbits of semigroups. This package \cite{GHSemigroup} is publicly available on GitHub, and also includes methods to computer-verify many of the claims in this paper (in particular, the function \texttt{runalltests}).

\subsection*{Acknowledgements} The authors would like to thank many people for helpful conversations, among them Alex Kontorovich, Carlos Matheus, Carlo Pagano, Mark Pollicott, Uri Shapira, Julia Slipantschuk, and Bianca Viray.  They are grateful to the referees for their insightful comments.  They are especially indebted to Mark Pollicott and Julia Slipantschuk for generously providing code for computing Hausdorff dimensions.  The authors are also grateful to the Centre international de rencontres math\'ematiques (CIRM) and its 2023 Semester II Chaire Jean-Morlet, Jayadev Athreya, for the opportunities for many of these conversations to take place.

\section{Statements of results}
\label{sec:results}

For a rational number $q$, its continued fraction expansion is
\[
q = [a_0; a_1, a_2, \ldots, a_n] 
:= a_0 + \cfrac{1}{a_1 +  \cfrac{1}{\ddots+\cfrac{1}{a_n}}},
\]
where $a_0\in\ZZ$, $a_i\in\ZZ^+$ for $1\leq i\leq n$. This expression is not unique: instead, there are exactly two ways to express any rational number as a continued fraction. Indeed, if $n=1$ or $a_n>1$, then
\[[a_0;a_1, a_2, \ldots, a_n] = [a_0; a_1, a_2, \ldots, a_n-1, 1].\]
In particular, each rational number has a unique continued fraction with an even length (i.e. with $n$ odd). Call this the \emph{even continued fraction} associated to $q$. By convention, we also permit $\frac{1}{0}=[0;0]$.

If $\Gamma\subseteq \SL(2, \ZZ)$, define $\Gamma^{\geq 0}:=\{\genmtx\in \Gamma:a,b,c,d\geq 0\}$. If $\Gamma$ is a group, then $\Gamma^{\geq 0}$ is a semigroup (in fact, it is a monoid, but the distinction is irrelevant in this paper).

Continued fractions of even type can be described\footnote{The notational choice of $L$ and $R$ corresponds to these matrices acting as ``going left'' or ``going right'' on Conway's topograph \cite{ConwaySensual, HatcherTopologyofNumbers} or Series' Farey tesselation \cite{SeriesModularSurfaceContinuedFractions}.} in terms of the semigroup:
\[
\SL(2, \ZZ)^{\ge 0} = \left\langle
\begin{pmatrix} 1 & 1 \\ 0 & 1 \end{pmatrix} 
,
\begin{pmatrix} 1 & 0 \\ 1 & 1 \end{pmatrix}
\right\rangle^+
=: \langle L, R \rangle^+.
\]
The notation $\langle A_1,A_2,\ldots \rangle^{+}$ indicates the monoid multiplicatively generated by the $A_i$, together with the identity $I$. If $\frac{x}{y}$ is a positive rational number in reduced terms, then
\begin{equation}\label{eqn:evencontfrac}
    \frac{x}{y} = [a_0; a_1, \ldots, a_{2n-1}] \Leftrightarrow \lmcol{x}{y}=L^{a_0}R^{a_1}L^{a_2}R^{a_3}\cdots L^{a_{2n-2}}R^{a_{2n-1}}\lmcol{1}{0}.
\end{equation}

We will discuss subsemigroups $\Gamma \subseteq \SL(2, \ZZ)^{\ge 0}$ and their orbits $\Gamma\smcol{x}{y}$. If $\smcol{u}{v}\in \Gamma\smcol{x}{y}$, call $u$ the numerator and $v$ the denominator.  The orbit of any such subsemigroup can therefore be interpreted as the collection of rational numbers whose continued fraction expansions are restricted in some way. Thus all of the results below have an interpretation in terms of continued fractions; we return to this language towards the end of the section.

Let $\kron{x}{y}$ denote the Kronecker symbol, which is defined for any pair of integers $x,y$; recall that it is $\pm 1$ if $x$ and $y$ are coprime, and $0$ otherwise.  If $\kron{x}{y} = -1$, then neither $x$ nor $y$ is a square.  
The starting point for our results is the observation that certain linear fractional transformations preserve Kronecker symbol.  For example, if $y$ is odd, the transformations $(x,y) \mapsto (x+4y,y)$ and $(x,y) \mapsto (x,4x+y)$ satisfy
\[
\kron{x}{y} = \kron{x+4y}{y} = \kron{x}{4x+y}
\]
by quadratic reciprocity.  As a consequence, the Kronecker symbol $\kron{x}{y}$ is constant on orbits of the semigroup generated by these two transformations.  Thus we have semigroup orbits such as
\[
\left\langle 
\begin{pmatrix} 1 & 0 \\ 4 & 1 \end{pmatrix},
\begin{pmatrix} 1 & 4 \\ 0 & 1 \end{pmatrix}
\right\rangle^+ \begin{pmatrix} 3 \\ 5 \end{pmatrix},
\]
which cannot support squares in either entry.  That squares are not possible in the first entry is clear a priori, as it is always $3 \pmod{4}$.  But it is possible to show that no congruence obstruction can account for the lack of squares in the second position.

Recall that the congruence subgroup $\Gamma_1(4)$ is defined as
\[
\Gamma_1(4) = \left\{ \genmtx  \in \SL(2,\ZZ) :\;  a\equiv d\equiv 1\Mod{4},\; c\equiv 0\Mod{4}\right\}.
\]
The most important subsemigroup we consider is the maximal semigroup which preserves the Kronecker symbol on rationals of odd denominator.

\begin{definition}
    Let $\Psi\subseteq \Gamma_1(4)^{\geq 0}$ be the set defined by
    \[\Psi:=\left\{\genmtx\in\Gamma_1(4)^{\geq 0}:\forall\; x,y\in\ZZ^{\geq 0}\text{ with $y$ odd},\text{ we have }\kron{x}{y}=\kron{ax+by}{cx+dy}\right\}.\]
\end{definition}

As the action of $\Gamma_1(4)^{\geq 0}$ preserves the parity of $y$ and nonnegativity of $x$ and $y$, $\Psi$ is a semigroup. 

It turns out that $\Psi$ has a nice characterization, which demonstrates that not only is it non-trivial, but it is in some sense half of $\Gamma_1(4)^{\geq 0}$.

\begin{proposition}[Proof in Section~\ref{sec:kronpres}]\label{prop:psicharacterization}
    The semigroup $\Psi$ is characterized as
    \[\Psi=\left\{\genmtx\in\Gamma_1(4)^{\geq 0}:\kron{a}{b}=1\right\}.\]
\end{proposition}

Besides the well-known observation that computation of Kronecker symbols can be accomplished as a version of the Euclidean algorithm, we are not aware of similar or related statements in the literature.  The only one the authors have been able to locate is \cite[Proposition 2]{Milovic}, which can be interpreted as giving a specific element of $\Psi$. As a statement about continued fraction expansions, we obtain a striking corollary concerning rationals of ``positive Kronecker symbol'':  the concatenation of their continued fraction expansions is again a rational of positive Kronecker symbol.  This is detailed in Corollary~\ref{cor:QR}.  

\begin{definition}
    We say that an orbit $\Gamma\smcol{x}{y}$ of a subsemigroup $\Gamma \subseteq \SL(2, \ZZ)$ exhibits a \emph{congruence obstruction modulo $n$} if the set of residues of numerators (or, respectively, denominators) modulo $n$ is not the set of all residues modulo $n$. If this happens for at least one orbit, we say $\Gamma$ exhibits congruence obstructions.
\end{definition}

The semigroup $\Psi$ has congruence obstructions modulo $4$:  for example, all numerators of $\Psi \smcol{1}{0}$ are $1 \pmod{4}$, and all denominators are $0 \pmod{4}$. On the other hand, numerators in the orbit $\Psi\smcol{2}{3}$ have no congruence obstructions. 

\begin{definition}
    We say that an orbit $\Gamma \smcol{x}{y}$ of a subsemigroup $\Gamma  \subseteq \SL(2, \ZZ)$ exhibits \emph{reciprocity obstructions} if the set $X$ of numerators (or, respectively, denominators) of $\Gamma \smcol{x}{y}$ satisfy
    \begin{enumerate}
        \item for all $n \ge 1$, the residues modulo $n$ of $X$ include squares; but
        \item $X$ contains no integer squares.
    \end{enumerate}
    If this happens for at least one orbit, we say $\Gamma$ exhibits reciprocity obstructions.
\end{definition}

In other words, a reciprocity obstruction is an obstruction to square values not induced by congruence obstructions.
In this paper we are restricting our attention to the numerator and denominator, but in another context, one might wish to extend the definitions of congruence and reciprocity obstructions to other linear functionals.  We also restrict our attention to squares for the present, but the future may bestow upon us other types of obstructions deserving of the name \emph{reciprocity} (as it did in the Apollonian case, \cite{HKRS23}).

Our main result about $\Psi$ is that it demonstrates reciprocity obstructions.

\begin{theorem}[Proof in Section~\ref{sec:psi-few}]\label{thm:semigroupreciprocity}
	Let $x, y$ be nonnegative coprime integers which satisfy one of the following conditions:
	\begin{itemize}
		\item $y$ is odd and $\kron{x}{y}=-1$;
		\item $(x, y)\equiv (1, 0)\pmod{4}$ and $\kron{x}{y}=-1$;
		\item $(x, y)\equiv (3, 0)\pmod{4}$, and $\kron{x}{y}=-\kron{-1}{y}$.
	\end{itemize}
	Then the numerators and denominators of the orbit $\Psi\smcol{x}{y}$ cannot be squares.
\end{theorem}

To conclude that certain orbits have reciprocity obstructions, it is necessary to classify the congruence obstructions for the orbit and check that they do not themselves already rule out squares; this happens sometimes, but not always.

\begin{theorem}[Proof in Section~\ref{sec:proof26}]
\label{thm:TableForPsi}
    Orbits of the semigroup $\Psi$ exhibit the congruence and reciprocity obstructions listed in Table~\ref{table:psi1obstructions}, and furthermore, the table represents a complete list of such. The result is also effective: let $X$ denote the set of numerators in the orbit $\Psi\smcol{x}{y}$ (respectively, denominators), and let $r$ denote the unique positive real root to the equation
    \[\frac{\sqrt{r}}{\log{r}\log{\log{r}}}= 7.542795 xy.\]
    Let $n$ be any integer that satisfies the corresponding congruence condition from Table~\ref{table:psi1obstructions}, and also satisfies either of the following conditions:
    \begin{itemize}
        \item $n\geq 2^k\max(r, 3.41\cdot 10^6)$ is not a square, where $2^k$ is the largest power of two at most $8xy$;
        \item $n\geq 8xy$ is a square, and there is no corresponding reciprocity obstruction listed in Table~\ref{table:psi1obstructions}.
    \end{itemize}
    Then $n$ is an element of $X$.
\end{theorem}

\begin{table}[htb]
	\centering
	\caption{Complete characterization of obstructions for $\Psi$.}\label{table:psi1obstructions}
\begin{tabular}{cccccc} 
\toprule
\multicolumn{2}{c}{Orbit type} & \multicolumn{4}{c}{Obstructions} \\
\toprule
    \multirow{2}{*}{$(x, y)\pmod{4}$} & \multirow{2}{*}{$\kron{x}{y}$}    & \multicolumn{2}{c}{Numerator} & \multicolumn{2}{c}{Denominator} \\ 
    &    & congruence  & reciprocity & congruence  & reciprocity \\ 
    \toprule
    \multirow{2}{*}{$(1, 0)$}     & $-1$             & \multirow{2}{*}{$\equiv 1\pmod{4}$} & $\neq n^2$ & \multirow{2}{*}{$\equiv 0\pmod{4}$} & $\neq n^2$ \\ 
                                  & $1$              &                              & -     &                              & -     \\ \midrule
    \multirow{2}{*}{\vspace{-0.25cm}$(3, 0)$}     & $-\kron{-1}{y}$  & \multirow{2}{*}{\vspace{-0.25cm}$\equiv 3\pmod{4}$} & -     & \multirow{2}{*}{\vspace{-0.25cm}$\equiv 0\pmod{4}$} & $\neq n^2$ \\ 
                                  & $\kron{-1}{y}$   &                              & -     &                              & -     \\ \midrule
    \multirow{2}{*}{$(\ast, 1)$}  & $-1$             & \multirow{2}{*}{-}        & $\neq n^2$ & \multirow{2}{*}{$\equiv 1\pmod{4}$} & $\neq n^2$ \\ 
                                  & $1$              &                              & -     &                              & -     \\ \midrule
                    $(\ast, 2)$   &                  & $\equiv 1\pmod{2}$                  & -     & $\equiv 2\pmod{4}$                  & -     \\ \midrule
    \multirow{2}{*}{$(\ast, 3)$}  & $-1$             & \multirow{2}{*}{-}        & $\neq n^2$ & \multirow{2}{*}{$\equiv 3\pmod{4}$} & -     \\ 
                                  & $1$              &                              & -     &                              & -     \\
\bottomrule
\end{tabular}
\end{table}

We wish to emphasize that these reciprocity obstructions are \emph{not} present in orbits of the Zariski closure $\SL(2)$, or even in $\Gamma_1(4)^{\ge 0}$.   The following is a statement of strong approximation (for background, see \cite{Rapinchuk}): 

\begin{proposition}\label{prop:introstrong}
Let $G = \Gamma_1(4)$ as a scheme over $\operatorname{Spec} \ZZ$.
    Then $\Psi$ is dense in $\Gamma_1(4)(\widehat{\ZZ})$.
\end{proposition}

\begin{proof}
    Note that modulo $2$, $\Psi$ and $\Gamma_1(4)$ coincide.  The rest is shown in Theorem~\ref{thm:strongapprox}.
\end{proof}

Observe that $\Gamma_1(4)$ has the same congruence obstructions as $\Psi$ but does \emph{not} have the further reciprocity obstructions of Theorem~\ref{thm:TableForPsi} (this can be verified directly).  Thus, the reciprocity obstructions in this paper are not inherited from any algebraic set.

That the reciprocity obstructions listed in Table~\ref{table:psi1obstructions} do occur for $\Psi$ is a simple application of Theorem~\ref{thm:semigroupreciprocity}, but that no other obstructions occur requires more work.

As Theorem~\ref{thm:TableForPsi} is effective, a full local-global result can be proven for specific orbits by combining the result with a computation. We do so in a simple case (this computation can be recreated with the code provided at \cite{GHSemigroup}).

\begin{theorem}[Proof in Section~\ref{sec:explicit}]\label{thm:psi123orbit}
The set of positive integers that are not numerators in the orbit $\Psi\smcol{2}{3}$ is equal to
\[\{n^2:n\in\ZZ^+\}\cup\{3, 6, 7, 10, 12, 15, 18, 19, 27, 31, 34, 55, 63, 99, 115\}.\]
\end{theorem}

Although the reciprocity obstructions in $\Psi$ are not a feature of $\Gamma_1(4)$, it is nevertheless quite a large semigroup.

\begin{proposition}[Proof in Section~\ref{sec:epsilon-dim}]
\label{prop:psidim1}
    The semigroup $\Psi$ is not finitely generated, and has Hausdorff dimension $1$.
\end{proposition}

It is interesting to consider subsemigroups which have Hausdorff dimension $1/2 < \delta < 1$: large enough that one might ask about the local-global principle, but small enough that such a question is challenging to answer. It has been shown that congruence obstructions can occur in such continued fraction semigroups $\Gamma$ of $\GL(2, \ZZ)$ \cite[pp. 2-3]{BourgainKontorovichZarembaAnnals}. 
We can generate subsemigroups with reciprocity obstructions by choosing generators from $\Psi$.

 \begin{theorem}[Proof in Section~\ref{sec:epsilon-dim}]
 \label{thm:epsilon-dim}
     Given any $\epsilon>0$, there exists a finitely generated subsemigroup $\Gamma\subseteq\SL(2, \ZZ)^{\geq 0}$ with Hausdorff dimension $1-\epsilon < \delta < 1$ that exhibits reciprocity obstructions.
 \end{theorem}

Two such subsemigroups of particular interest are as follows.

\begin{definition}
    Define
    \[\Psi_1:=\left\langle\lm{1}{1}{0}{1},\lm{1}{0}{4}{1}\right\rangle^+,\qquad\Psi_2:=\left\langle\lm{1}{4}{0}{1},\lm{1}{0}{4}{1}\right\rangle^+.\]
\end{definition}

We estimate the respective Hausdorff dimensions of their limit sets as $\delta_1\approx 0.719$ and $\delta_2=0.6045578\pm 10^{-7}$, both of which exceed $1/2$. One example of reciprocity obstruction for each of $\Psi_1$ and $\Psi_2$ is as follows.

\begin{proposition}[Proof in Section~\ref{sec:twoexamples}]\label{prop:twoexamples}
    Numerators in the orbit $\Psi_1\smcol{2}{3}$ have no congruence obstructions, yet cannot be square. Denominators in the orbit $\Psi_2\smcol{3}{8}$ are only restricted to be $0\pmod{4}$, yet cannot be square. In particular, these are both instances of reciprocity obstructions.
\end{proposition}

We show that Table~\ref{table:psi1obstructions} is also valid for $\Psi_1$.

\begin{proposition}[Proof in Section~\ref{sec:twoexamples}]
\label{prop:psi1table}
    Let $x,y$ be coprime nonnegative integers. Then the congruence obstructions described in Table~\ref{table:psi1obstructions} still hold for the orbit $\Psi_1\smcol{x}{y}$. Furthermore, if a reciprocity obstruction is listed, then it occurs in the orbit $\Psi_1\smcol{x}{y}$ as well.
\end{proposition}

While we were able to show that Table~\ref{table:psi1obstructions} completely characterized the orbits of $\Psi$, we are not able to do so for $\Psi_1$. Instead, we collect computational evidence towards this result, and formulate the following conjecture.

\begin{conjecture}\label{conj:psi1localglobal}
    Table \ref{table:psi1obstructions} completely characterizes the numbers that can appear as numerators or denominators in orbits of $\Psi_1$. In other words, every sufficiently large integer that is not ruled out by the given congruence and reciprocity obstructions will appear in the orbit.
\end{conjecture}

Based on computational evidence, an example explicit conjecture is the following.

\begin{conjecture}\label{conj:twoexamples}
    Every non-square integer $n > 10569$ appears as a numerator in the orbit $\Psi_1\smcol{2}{3}$.
\end{conjecture}

Finally, we turn to statements in terms of continued fractions, and compare our results to Zaremba's conjecture. The orbits of both $\Psi_1$ and $\Psi_2$ have interesting descriptions in terms of continued fractions.

\begin{proposition}[Proof in Section~\ref{sec:ctd}]\label{prop:psi12incontfrac}
    Let $x,y$ be coprime nonnegative integers, where the even continued fraction expansion of $\frac{x}{y}$ is $[0;a_1,a_2,\ldots,a_{2n-1}]$. For coprime nonnegative integers $u, v$, we have $\smcol{u}{v}\in\Psi_2\smcol{x}{y}$ if and only if the even continued fraction of $\frac{u}{v}$ has the form
    \[
    \frac{u}{v}=
        [4b_0;4b_1, 4b_2, \ldots, 4b_{2m}, 4b_{2m+1} + a_1, a_2, \ldots, a_{2n-1}],
    \]
    where the $b_i$ are nonnegative integers with $b_i>0$ if $1\leq i\leq 2m$.
    Similarly, $\smcol{u}{v}\in\Psi_1\smcol{x}{y}$ if and only if the even continued fraction of $\frac{u}{v}$ has the form
    \[
    \frac{u}{v}=
        [b_0;4b_1, b_2, 4b_3,\ldots, b_{2m-2}, 4b_{2m-1}, b_{2m}, 4b_{2m+1} + a_1, a_2, \ldots, a_{2n-1}],
    \]
    where the $b_i$ are nonnegative integers with $b_i>0$ if $1\leq i\leq 2m$.
\end{proposition}

Stated more informally, for $\Psi_2$, the continued fraction is built from the alphabet $\{4, 8, 12, 16, \ldots\}$, with the tail being the continued fraction of $\frac{u}{v}$. For $\Psi_1$, we have a similar description, except only every other term must be a multiple of 4.  This allows us to describe denominators of rationals built from the alphabet $4\ZZ^+$ in terms of an orbit of $\Psi_2$.

\begin{corollary}[Proof in Section~\ref{sec:ctd}]\label{cor:psi2and0mod4}
    Let $S$ be the set of denominators of rational numbers in $(0, 1)$ which have continued fractions that can be built from the alphabet $4\ZZ^+$. Then $S$ is the union of the numerators and denominators of the orbit $\Psi_2\smcol{0}{1}$.
\end{corollary}

Since orbits of $\Psi_2$ exhibit reciprocity obstructions, this shows that such obstructions exist in a context with the flavour of Zaremba's conjecture. However, Zaremba's conjecture as stated remains untouched: it deals with continued fractions built from a finite alphabet $\mathcal{A}=\{1, 2, \ldots, A\}\not\subseteq 4\ZZ^+$, as well as the orbit of $\smcol{0}{1}$.  These do not exhibit reciprocity obstructions of the type we study.

We give a more explicit example class of orbits with reciprocity obstructions.

\begin{theorem}[Proof in Section~\ref{sec:ctd}]\label{thm:zarembareciprocity}
    Let $\mathcal{A}$ be a subset of $4\ZZ^+$ containing $4$ and $8$, and let $\Psi_{\mathcal{A}} \subseteq \GL(2,\ZZ)$ be the semigroup corresponding to continued fractions built from this alphabet, i.e.
    \[
     \Psi_\mathcal{A} = \langle \sm{0}{1}{1}{a} : a \in \mathcal{A} \rangle^+.
    \]
    Then $\Psi_{\mathcal{A}}\smcol{x}{y}$ exhibits reciprocity obstructions in the denominator if $y\equiv 1\pmod{4}$, $x$ is odd, and $\kron{x}{y}=-1$.
\end{theorem}

Requiring the alphabet to contain $4,8$ ensures the congruence obstructions admit squares, giving rise to a reciprocity obstruction. One example of such an orbit is $\Psi_\mathcal{A} \smcol{3}{5}$. The congruence condition on the denominator is $1\pmod{2}$, which admits squares, but no square denominators occur. Since $3/5 = [0;1,1,2]$, when stated in terms of continued fractions, we obtain the following result.

\begin{corollary}[Proof in Section~\ref{sec:ctd}]\label{cor:psi2threefiveorbitcontfracwhydoesjamesusesuchlongnames}
    Let $\mathcal{A}$ be a subset of $4\ZZ^+$ containing $4$ and $8$. 
    
    Consider the set $S$ of rational numbers which have a continued fraction expansion of the form
    \[
    [0; a_1, a_2, \ldots, a_n,1,1,2],
    \]
    where $a_i\in\mathcal{A}$. Then all positive odd numbers are admissible as denominators for elements of $S$, yet no denominator is a square.

Equivalently, consider the set $S'$ of rational numbers which have a continued fraction expansion of the form
    \[
    [0; a_1, a_2, \ldots, a_n],
    \]
    where $a_i\in\mathcal{A}$. Let $L(x/y) = 3x+5y$.  Then all positive odd numbers are admissible for $L(S')$, yet $L(S')$ contains no squares.
\end{corollary}

Which of the two phrasings one prefers is a matter of taste, but the second phrasing is an example subject to the Local-Global Conjecture for Continued Fraction Alphabets, so we immediately obtain the following result.

\begin{theorem}
\label{thm:conjfalse}
    The Local-Global Conjecture for Continued Fraction Alphabets is false.
\end{theorem}

The smallest subalphabet of $4\ZZ^+$ for which the Hausdorff dimension exceeds $1/2$ is $\mathcal{A} = \{4,8,12,\ldots,128\}$, where $\Psi_\mathcal{A}$ has Hausdorff dimension $\delta_{\mathcal{A}}=0.500890640842\pm 10^{-12}>0.5$ (many thanks go to Julia Slipantschuk, who shared with us her code to rigorously compute Hausdorff dimensions for finite alphabets; see \cite{pollicott2023effective, bandtlow2020lagrange}).

Finally, we can combine orbits to find a larger family of continued fractions which miss squares.

\begin{corollary}[Proof in Section~\ref{sec:ctd}]\label{cor:bigcontfracnosquares}
    Let $S$ be the set of rational numbers which have a continued fraction expansion of the form
    \[
    [0; 4a_1, 4a_2, \ldots, 4a_n,a_{n+1},1,2],
    \]
    where the $a_i$ are positive integers. Then all positive integers are admissible as denominators for elements of $S$, yet no denominator is a square.
\end{corollary}

A computation suggests the following conjecture.

\begin{conjecture}\label{conj:contfraceventuallyall}
    For all non-square integers $v> 7968219670470$, there exists a coprime positive integer $u<v$ for which
    \[\frac{u}{v}=[0; 4a_1, 4a_2, \ldots, 4a_n,a_{n+1},1,2],\]
    where the $a_i$ are positive integers.
\end{conjecture}

In Section \ref{sec:kronpres}, we cover some technical results concerning Kronecker symbols and their behaviour under M\"obius transformation, which proves Proposition~\ref{prop:psicharacterization}.  These results rule out squares in certain orbits of $\Psi$ without too much work, but to claim an interesting obstruction, we intend to show everything else sufficiently large will appear.  To this end, the very technical Section~\ref{sec:psi-many} applies the results of Section~\ref{sec:kronpres} to give sufficient conditions for integers to appear in orbits of $\Psi$; and in turn, we can then use these to give complete orbits in the following Section~\ref{sec:psi-few}, where the main results for $\Psi$ (namely, Theorems~\ref{thm:semigroupreciprocity} and \ref{thm:TableForPsi}) are proven. The explicit example of Theorem~\ref{thm:psi123orbit} is covered in Section~\ref{sec:explicit}. Subsemigroups of smaller Hausdorff dimension are covered in Sections~\ref{sec:thin} and \ref{sec:psi12}, and consequences for continued fractions are addressed in Section~\ref{sec:ctd}.  Some remarks related to finding sub\emph{groups} of $\SL(2, \ZZ)$ exhibiting reciprocity obstructions are considered in Section \ref{section:groups}.

\section{Behaviour of the Kronecker symbol under M\"{o}bius transformation}\label{sec:kronpres}

In this section, we prove Proposition~\ref{prop:psicharacterization}. Along the way, we develop the main machine (Proposition~\ref{prop:mobiuskron}), which is an expression for $\kron{ax+by}{cx+dy}$ in terms of $\kron{x}{y}$, and is used extensively throughout the paper.

Recall from the introduction that
    \begin{equation*}
    \Psi :=\left\{\genmtx\in\Gamma_1(4)^{\geq 0}:\forall x,y\in\ZZ^{\geq 0}\text{ with $y$ odd},\text{ we have }\kron{x}{y}=\kron{ax+by}{cx+dy}\right\}.
     \end{equation*}
As observed there, this is a sub-semigroup of $\Gamma_1(4)^{\ge 0}$.  In order to make it explicit, however, we wish to show that, equivalently,
\[\Psi=\left\{\genmtx\in\Gamma_1(4)^{\geq 0}:\kron{a}{b}=1\right\}.\]
This is the statement of 
Proposition~\ref{prop:psicharacterization}.  We begin by showing that, in $\Gamma_1(4)^{\ge 0}$, the Kronecker symbols of rows and columns all agree.

\begin{lemma}\label{lem:allkronequal}
	Let $\genmtx\in\Gamma_1(4)^{\geq 0}$. Then
	\[\kron{a}{b}=\kron{c}{d}=\kron{a}{c}=\kron{b}{d}.\]
\end{lemma}
\begin{proof}
Since all entries are nonnegative, the Kronecker symbol is multiplicative in the numerator, and $\kron{4n+m}{n} = \kron{m}{n}$.  Furthermore, if $n \not\equiv 2 \pmod{4}$, then $\kron{n+m}{n} = \kron{m}{n}$. If at least one of $m,n$ is $1\pmod{4}$, then quadratic reciprocity implies that $\kron{m}{n}=\kron{n}{m}$.
We compute:
\[
	\kron{a}{b}=\kron{d}{b}\kron{ad}{b}=\kron{d}{b}\kron{bc+1}{b}=\kron{d}{b}=\kron{b}{d},
\]
noting that $4b \mid bc$, and $d\equiv 1\pmod{4}$.
Similarly, $\kron{a}{c}=\kron{c}{d}$. Finally,
\[\kron{b}{d}
 =\kron{c}{d}\kron{bc}{d}=\kron{c}{d}\kron{ad-1}{d}=\kron{c}{d}\kron{-1}{d}=\kron{c}{d},
\]
since $d\equiv 1\pmod{4}$.
\end{proof}

To prove Proposition~\ref{prop:psicharacterization}, it suffices to show that if $\genmtx \in \Gamma_1(4)^{\ge 0}$, then
\[
\kron{ax+by}{cx+dy} = \kron{a}{b}\kron{x}{y},
\]
when $x,y$ are coprime nonnegative integers with $y$ odd.

We will give a brief demonstration of this property in one case, before giving the full technical details. To this end, suppose for simplicity that $(x,y) \equiv (0,1) \pmod{4}$ and $\gcd(x,d)=1$. In particular, we have
\[
c, x \equiv 0 \pmod{4}, \quad
a, d, y, cx+dy \equiv 1 \pmod{4},
\]
so that a Kronecker symbol involving an entry from the second list $a, d, y, cx+dy$ will take the same value when its entries are swapped, by quadratic reciprocity. In addition, a Kronecker symbol with any of these entries is ``modular'' in the other entry, i.e., $\kron{m+n}{n} = \kron{m}{n}=\kron{m}{m+n}$, and they are multiplicative in either entry. Finally, note that $d(ax+by) - b(cx+dy) = x$. Then,
\begin{align*}
\kron{ax+by}{cx+dy}
&=
\kron{d}{cx+dy}\kron{d(ax+by)}{cx+dy} \\
&=
\kron{d}{cx+dy}\kron{d(ax+by)-b(cx+dy)}{cx+dy} \\
&=
\kron{d}{cx+dy}\kron{x}{cx+dy} 
=
\kron{d}{cx}\kron{x}{dy} \\
&=
\kron{d}{c}\kron{d}{x}\kron{x}{d} \kron{x}{y} = 
\kron{a}{b}\kron{x}{y}.
\end{align*}

The next proposition gives the action of $\SL(2, \ZZ)^{\geq 0}$ on the Kronecker symbol in greatest generality. It follows the same outline as the above computation.

\begin{proposition}\label{prop:mobiuskron}
	Define $\opart{n}$ to be the odd part of $n$. Let $\genmtx\in\SL(2, \ZZ)^{\geq 0}$, let $x, y$ be nonnegative integers, and assume that $\gcd(x, d)=1$. Let
	\[
	A=\frac{\opart{x}-1}{2},\quad B=\frac{\opart{d}-1}{2},\quad C=\frac{\opart{(cx+dy)}-1}{2},\quad D=\frac{\opart{y}-1}{2},\quad \alpha = AB+AC+BC+AD,
	\]
	and let
	\[
	\mu_1 = \begin{cases} 
		\kron{cdxy+1}{2} & \text{if $2\mid\mid d$ or $2\mid\mid x$;}\\
		1 & \text{else;}
		\end{cases}
	\quad \mu_2 = \begin{cases}
		\kron{bx(cx+dy)+1}{2} & \text{if $2\mid\mid cx+dy$;}\\
		1 & \text{else.}
	\end{cases}
	\]
	Then
	\[\kron{ax+by}{cx+dy}=(-1)^{\alpha}\mu_1\mu_2\kron{c}{d}\kron{x}{y}.\]
\end{proposition}
\begin{proof}
	If $x, y$ are not coprime, the result is immediate. Otherwise, as $\gcd(x, d)=1$, $\gcd(d, cx+dy)=1$. Then
	\begin{multline*}
		\kron{ax+by}{cx+dy}=\kron{d}{cx+dy}\kron{adx+bdy}{cx+dy}=\epsilon_1\kron{cx+dy}{d}\epsilon_2\kron{adx+bdy-b(cx+dy)}{cx+dy}\\
		=\epsilon_1\epsilon_2\epsilon_3\kron{cx}{d}\kron{x}{cx+dy},
	\end{multline*}
	where, by quadratic reciprocity, $\epsilon_1=(-1)^{BC}$, and $\epsilon_2$ and $\epsilon_3$ are as follows.    If $cx+dy\not\equiv 2\pmod{4}$, then $\epsilon_2=1$. Otherwise, $x$ is necessarily odd, and we have\footnote{Recall that if $n$ is odd, $\kron{kn+m}{2n} = \kron{kn+m}{2}\kron{kn+m}{n} = \kron{kn+m}{2}\kron{m}{n} = \kron{kn+m}{2}\kron{m}{2}\kron{m}{2n}$.}
	\[
	\epsilon_2=\kron{adx+bdy}{2}\kron{x}{2}=\kron{(bc+1)x+bdy}{2}\kron{x}{2}=\kron{bx(cx+dy)+x^2}{2}=\mu_2,
	\]
	since $x^2\equiv 1\pmod{8}$ and $\kron{\cdot}{2}$ is well defined modulo $8$. Similarly, the term $\epsilon_3$ is $1$ if $d\not\equiv 2\pmod{4}$. Otherwise,
	\[\epsilon_3=\kron{cx+dy}{2}\kron{cx}{2}=\kron{cdxy+(cx)^2}{2}=\kron{cxdy+1}{2}=\mu_1,\]
	as $cx$ is necessarily odd (it is coprime to $d$). Moving forward, 
	\[
	\kron{cx}{d}\kron{x}{cx+dy}=\kron{c}{d}\kron{x}{d}\epsilon_4\kron{cx+dy}{x}=\epsilon_4\kron{c}{d}\kron{x}{d}\epsilon_5\kron{dy}{x}=\epsilon_4\epsilon_5\kron{c}{d}\kron{x}{d}\kron{d}{x}\kron{y}{x}.
	\]
	It is immediate that $\epsilon_4=(-1)^{AC}$, and $\epsilon_5=1$ if $x\not\equiv 2\pmod{4}$. Otherwise,
	\[\epsilon_5=\kron{cx+dy}{2}\kron{dy}{2}=\kron{cdxy+(dy)^2}{2}=\kron{cxdy+1}{2}=\mu_1,\]
	since $dy$ is necessarily odd (using $\gcd(x,y)=1$). Note that we can combine $2\mid\mid x$ with $2\mid\mid d$ in the statement since both conditions cannot occur simultaneously.
	
	Finally,
	\[
	\kron{c}{d}\kron{x}{d}\kron{d}{x}\kron{y}{x}=\kron{c}{d}\epsilon_6\epsilon_7\kron{x}{y},
	\]
	where $\epsilon_6=(-1)^{AB}$ and $\epsilon_7=(-1)^{AD}$. Combining this all together gives the proposition.
\end{proof}

While Proposition \ref{prop:mobiuskron} has complicated conditions, they often collapse when certain modulo 4 conditions are added to some of $a, b, c, d$. Code to computationally check this formula is found in \cite{GHSemigroup}.

To prove Proposition~\ref{prop:psicharacterization}, we would like an analogous theorem where we do not have to assume that $\gcd(x, d)=1$; the rest of the section is devoted to this.

\begin{lemma}\label{lem:krondenomshift}
	Let $x,y$ be positive coprime integers and let $k$ be a nonnegative integer.  Suppose either $x\not\equiv 3\pmod{4}$ or $v_2(k)\geq v_2(y)$, where $v_2(n)$ denotes the largest power of 2 dividing $n$. Then
	\[\kron{x}{y}=\kron{x}{y+4kx}.\]
\end{lemma}
\begin{proof}
	From quadratic reciprocity, we have
	\[\kron{x}{y+4kx}=(-1)^{\alpha_1}\kron{y+4kx}{x}=(-1)^{\alpha_1}\kron{y}{x}=(-1)^{\alpha_1+\alpha_2}\kron{x}{y},\]
	where
	\[\alpha_1=\frac{\opart{x}-1}{2} \cdot \frac{\opart{(y+4kx)}-1}{2},\quad \alpha_2=\frac{\opart{x}-1}{2} \cdot \frac{\opart{y}-1}{2}.\]
	If $x\equiv 1\pmod{4}$ then we are done.  Since $y+4kx\equiv y\pmod{4}$, if $y$ is odd, then $\alpha_1=\alpha_2$, and we are again done. This leaves the case of $x\equiv 3\pmod{4}$. Since $v_2(k)\geq v_2(y)$, the odd parts of $y$ and $y+4kx$ are equivalent modulo $4$, which completes the proof.
\end{proof}

We can now show that the action of $\Gamma_1(4)^{\geq 0}$ on Kronecker symbols is relatively well-behaved.  This will be used to prove Proposition~\ref{prop:psicharacterization} and Theorem~\ref{thm:semigroupreciprocity}.

\begin{lemma}\label{lem:kronpreserved}
	Let $\genmtx\in\Gamma_1(4)^{\geq 0}$ and let $x,y\in\ZZ^{\geq 0}$ be coprime integers, where $y\not\equiv 2\pmod{4}$ or $b\equiv 0\pmod{4}$. Then
	\[\kron{ax+by}{cx+dy}=(-1)^{\alpha}\kron{c}{d}\kron{x}{y},\text{ where } \alpha=\frac{\opart{x}-1}{2}\left(\frac{\opart{y}-1}{2}+\frac{\opart{(cx+dy)}-1}{2}\right).\]
	In particular, if $y$ is odd or $(x, y)\equiv (1, 0)\pmod{4}$, then $\kron{ax+by}{cx+dy}=\kron{a}{b}\kron{x}{y}$.
\end{lemma}
\begin{proof}
	If $\gcd(x, d) = 1$, then the first statement follows immediately from Proposition \ref{prop:mobiuskron}.
	
	In general, we will show that we can always reduce to this case.	Consider replacing $\genmtx$ by $\sm{a}{b}{c+2^rka}{d+2^rkb}\in\Gamma_1(4)^{\geq 0}$ for some non-negative integer $k$ and a fixed integer $r\geq 2$. Choose $r$ large enough so that $\opart{(cx+dy)}\equiv \opart{((c+2^rka)x+(d+2^rkb)y)}\pmod{4}$, and $r\geq v_2(cx+dy) + 2$ if $ax+by\equiv 3\pmod{4}$. By Lemma \ref{lem:krondenomshift}, we have
	\[\kron{ax+by}{cx+dy}=\kron{ax+by}{(c+2^rka)x+(d+2^rkb)y}.\]
	Since $\sm{a}{b}{c+2^rka}{d+2^rkb}\in\Gamma_1(4)^{\geq 0}$, Lemma~\ref{lem:allkronequal} gives
	\[\kron{c+2^rka}{d+2^rkb}=\kron{a}{b}=\kron{c}{d},\]
	so it suffices to prove the proposition for this new matrix. As $d$ is odd, $\gcd(d, 2^rb)=1$, so there exists a non-negative integer $k$ such that $d+k(2^rb)$ is coprime to $x$. By using this value of $k$, we have reduced to the case of $\gcd(x, d)=1$, completing the proof.
	
	The final sentence follows from Lemma \ref{lem:allkronequal} and noting that $\alpha$ is even in the listed cases.
\end{proof}

Proposition~\ref{prop:psicharacterization} follows as a corollary, since if $y$ is odd, we have
\begin{equation}
\label{eqn:kronchange}
\kron{ax+by}{cx+dy}=\kron{a}{b}\kron{x}{y},
\end{equation}
which is always equal to $\kron{x}{y}$ if and only if $\kron{a}{b}=1$.

There are some interesting consequences of Proposition~\ref{prop:psicharacterization} in the language of continued fractions.

\begin{corollary}\label{cor:QR}
Let $S \subseteq (0,1)$ be the subset of rational numbers $s/t$ in lowest form such that $\kron{s}{t} = 1$, and $(s,t)\equiv (1,0)\pmod{4}$. Then if $[0;a_1,a_2,\ldots,a_{2m-1}]$ and $[0;b_1,b_2,\ldots,b_{2n-1}]$ are the even continued fraction expansions for two elements of $S$, then the concatenation $[0;a_1,a_2,\ldots,a_{2m-2}, a_{2m-1}+b_1,b_2,\ldots,b_{2n-1}]$ is the even continued fraction expansion for some element of $S$.
\end{corollary}

\begin{proof}
    Let $\frac{s_1}{t_1} = [0; a_1, a_2, \ldots, a_{2m-1}], \frac{s_2}{t_2} = [0; b_1, b_2, \ldots, b_{2m-1}] \in S$, and define \[
    M_1=L^0R^{a_1}L^{a_2}\cdots R^{a_{2m-1}}, \quad
    M_2=L^0R^{b_1}L^{b_2}\cdots R^{b_{2n-1}}.
    \]
    Since $\smcol{s_i}{t_i}=M_i\smcol{1}{0}$, we have $M_i=\sm{s_i}{b_i}{t_i}{d_i}\in\Psi$ by the definition of $S$. Thus $M_1M_2 \in \Psi$.
    The rational number formed by the first column is thus an element of $S$.
\end{proof}

\section{Sufficient conditions for integers to be represented by orbits of $\Psi$}
\label{sec:psi-many}

In this section, we give a series of lemmas providing sufficient conditions for integers to appear in orbits of $\Psi$.  The results depend on the Kronecker symbol properties of $\Psi$ from the previous section.  The general form of the conditions is as follows.  Consider a non-negative integer $n \equiv y \pmod{4}$ (respectively $x \pmod{\gcd(y, 4)}$).  Then $n$ will appear as a denominator (respectively, numerator) in $\Psi \smcol{x}{y}$ if the following holds:  any sequence of $\lfloor\frac{n}{8xy}\rfloor$ consecutive integers contains terms which have Kronecker symbols of both values $1$ and $-1$ modulo the odd part $\opart{n}$ of $n$.  This type of sufficient condition will typically be satisfied for large enough $n$, which we will exploit in the next section.

\begin{figure}
\begin{pspicture}(-0.5,-0.45)(8.25,5.5)
  \psaxes[ticks = none, labels = none]{->}(0,0)(-0.2,-0.2)(8.5,5.5)[$u$,0][$v$,90]
  \psyTick(0){0}
  \psxTick(0){0}
  \psline(0,4.9)(8.16666,0) 
     \psline(0,2.8)(4.6666,0) 
 \multirput(0,0)(0,.5){11}{\psdot[dotstyle=o,fillcolor=lightgray]}
 \multirput(0.5,0)(0,.5){11}{\psdot[dotstyle=o,fillcolor=lightgray]}
 \multirput(1.5,0)(0,.5){11}{\psdot[dotstyle=o,fillcolor=lightgray]}
 \multirput(1.0,0)(0,.5){11}{\psdot[dotstyle=o,fillcolor=lightgray]}
 \multirput(1.5,0)(0,.5){11}{\psdot[dotstyle=o,fillcolor=lightgray]}
 \multirput(2.0,0)(0,.5){11}{\psdot[dotstyle=o,fillcolor=lightgray]}
 \multirput(2.5,0)(0,.5){11}{\psdot[dotstyle=o,fillcolor=lightgray]}
 \multirput(3.0,0)(0,.5){11}{\psdot[dotstyle=o,fillcolor=lightgray]}
 \multirput(3.5,0)(0,.5){11}{\psdot[dotstyle=o,fillcolor=lightgray]}
 \multirput(4.0,0)(0,.5){11}{\psdot[dotstyle=o,fillcolor=lightgray]}
 \multirput(4.5,0)(0,.5){11}{\psdot[dotstyle=o,fillcolor=lightgray]}
 \multirput(5.0,0)(0,.5){11}{\psdot[dotstyle=o,fillcolor=lightgray]}
 \multirput(5.5,0)(0,.5){11}{\psdot[dotstyle=o,fillcolor=lightgray]}
 \multirput(6.0,0)(0,.5){11}{\psdot[dotstyle=o,fillcolor=lightgray]}
 \multirput(6.5,0)(0,.5){11}{\psdot[dotstyle=o,fillcolor=lightgray]}
 \multirput(7.0,0)(0,.5){11}{\psdot[dotstyle=o,fillcolor=lightgray]}
 \multirput(7.5,0)(0,.5){11}{\psdot[dotstyle=o,fillcolor=lightgray]}
  \multirput(8.0,0)(0,.5){11}{\psdot[dotstyle=o,fillcolor=lightgray]}
 \uput[180](0,5){$n=49$}
 \uput[180](0,3){$n=28$}
 \uput[45](0.4,2.4){$\kron{1}{5}$}
 \psdot[dotstyle=o,fillcolor=white](0.5,2.5)
 \uput[45](1.4,3.9){$\kron{3}{8}$}
  \psdot[dotstyle=o,fillcolor=black](1.5,4.0)
 \uput[45](2.9,0.9){$\kron{6}{2}$}
  \psdot[dotstyle=o,fillcolor=black](3.0,1.0)
 \uput[45](3.9,2.4){$\kron{8}{5}$}
  \psdot[dotstyle=o,fillcolor=black](4.0,2.5)
 \uput[45](6.4,0.9){$\kron{13}{2}$}
 \psdot[dotstyle=o,fillcolor=black](6.5,1.0)
\end{pspicture}
\caption{The upper diagonal line is $n=3u + 5v$ for $n=49$; the lower line is $n = 3u + 5v$ for $n=28$.  Integer points are marked.  A white-filled dot indicates that the point gives rise to a numerator $n$ in the orbit $\Psi\smcol{3}{5}$.  Since this orbit avoids squares, there are no white dots on the line for $n=49$.}
\label{fig:repcond}
\end{figure}
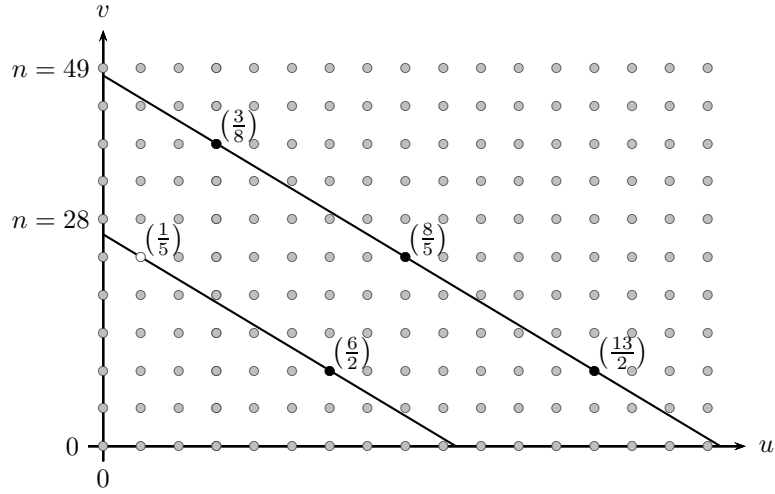

We begin by stating an explicit equivalent condition for an integer to appear.

\begin{lemma}\label{lem:repcond}
    Let $x, y$ be nonnegative coprime integers.	A positive integer $n$ appears as a numerator in the orbit $\Psi\smcol{x}{y}$ if and only if there exist nonnegative integers $u, v$ such that
	\[u\equiv 1\pmod{4}, \quad ux+vy=n,\quad \kron{u}{v}=1.\]
	Similarly, $n$ appears as a denominator if and only if there exist nonnegative integers $u, v$ such that
	\[u\equiv 0\pmod{4}, \quad v\equiv 1\pmod{4},\quad ux+vy=n,\quad \kron{u}{v}=1.\]
\end{lemma}
\begin{proof}
	If $n$ is a numerator, there is a matrix $\genmtx\in\Psi$ for which $\genmtx\smcol{x}{y}=\smcol{n}{m}$ for some integer $m$. Then $(u, v)=(a, b)$ gives the claimed pair.
	
	For the other direction of the first (numerator) statement, assume such a pair $(u, v)$ exists. As $\kron{u}{v}=1$, they are coprime, so there exists integers $c,d$ for which $du-cv=1$. By replacing $c$ by $c+uk$ and $d$ by $d+vk$ for a large positive integer $k\equiv -c\pmod{4}$, we can assume that $c, d$ are non-negative and $c\equiv 0\pmod{4}$. Since $u \equiv 1 \pmod{4}$, it follows that $d\equiv 1\pmod{4}$, hence $\sm{u}{v}{c}{d}\in\Psi$, showing that $n$ is a numerator of the orbit.
	
	The denominator case follows analogously, noting that there is now a residue condition on $v$ as well.
\end{proof}

Lemma~\ref{lem:repcond} is illustrated in Figure~\ref{fig:repcond}.  As $n$ grows, the number of integer points in the first quadrant on the line $n = ux + vy$ increases:  for $n$ to appear as a numerator or denominator, at least one point must satisfy certain congruence and Kronecker conditions.  As $n$ grows, we expect this to become more common.
The easiest case of Lemma~\ref{lem:repcond} to analyze is if either $x=0$ or $y=0$, which follows immediately.

\begin{corollary}\label{cor:psiorbitwithzero}
    The following are true:
    \begin{itemize}
        \item The set of numerators of $\Psi\smcol{0}{1}$ is the set of all positive integers;
        \item The set of denominators of $\Psi\smcol{0}{1}$ is the set of all positive integers equivalent to $1\pmod{4}$;
        \item The set of numerators of $\Psi\smcol{1}{0}$ is the set of all positive integers equivalent to $1\pmod{4}$;
        \item The set of denominators of $\Psi\smcol{1}{0}$ is the set of all positive integers equivalent to $0\pmod{4}$.
    \end{itemize}
\end{corollary}

Assuming $x,y>0$, we want to use Lemma~\ref{lem:repcond} to produce sufficient conditions. We start with the easier case of the denominator.

\begin{lemma}\label{lem:denomeventual}
	Let $x$ and $y$ be coprime positive integers, and let $n$ be a positive integer satisfying $n\equiv y\pmod{4}$. Furthermore, assume that at least one of the following conditions holds:
     \begin{itemize}
         \item $n\geq 8xy$ and $n=2m^2$ for some integer $m$;
         \item in any consecutive sequence of $\lfloor\frac{n}{8xy}\rfloor$ integers, there is an integer $r_-$ with $\kron{r_-}{\opart{n}}=-1$ and an integer $r_{+}$ with $\kron{r_{+}}{\opart{n}}=1$.
     \end{itemize}
    Then $n$ is a denominator in the orbit $\Psi\smcol{x}{y}$.
\end{lemma}

\begin{proof}
The overall plan of the proof is to range through the solutions $(u(k),v(k))$ to the linear equation of Lemma~\ref{lem:repcond} by parametrizing in terms of a variable $k$, and then use Proposition~\ref{prop:mobiuskron} to evaluate whether the Kronecker symbol is $1$.  For this purpose, we define a matrix $\sm{x}{y}{c}{d}$ to apply to $\smcol{u(k)}{v(k)}$, so that the Proposition describes the Kronecker symbol $\kron{u(k)}{v(k)}$ in terms of a symbol with fixed denominator, and numerator ranging through an arithmetic progression controlled by $k$. 
 The trick is that $u(k)x+v(k)y=n$, a constant in terms of $k$.

\textbf{Setting up the list of solutions to $n=ux+vy$ with $(u,v) \equiv (0,1) \pmod{4}$.}
As $n\equiv y\pmod{4}$, there exists integers $(u, v)$ satisfying $ux+vy=n$ with $u \equiv 0 \pmod 4$ and $v\equiv 1\pmod{4}$. Pick the solution with smallest nonnegative value of $u$, say $(u_0, v_0)$. Then $(u, v) = (u_0+4yk, v_0-4xk)$ for $k\in\ZZ$ gives the general solution to the system of equations $ux+vy=n$, $u \equiv 0 \pmod 4$, $v\equiv 1\pmod{4}$.

In particular,
\begin{itemize}
	\item $0\leq u_0 < 4y$;
	\item $(u_0,v_0)$ is a nonnegative pair if and only if $0\leq k \leq \frac{n}{4xy}-\frac{u_0}{4y}$.
\end{itemize}
Lemma~\ref{lem:repcond} implies that $n$ is a denominator of the orbit $\Psi\smcol{x}{y}$ if and only if there exists an integer $k$ with
\begin{equation}
\label{eqn:krange}
0\leq k \leq \frac{n}{4xy}-\frac{u_0}{4y}\quad\text{and}\quad\kron{u}{v}=\kron{u_0+4yk}{v_0-4xk}=1.
\end{equation}

\textbf{Finding a suitable matrix for Proposition~\ref{prop:mobiuskron}.}
Next, fix a pair of integers $(c_0, d_0)$ such that 
\begin{equation}
    \label{eqn:xd0yc0den}
    xd_0-yc_0=1. 
\end{equation}
For any integer $s$, we have $M=\sm{x}{y}{c_0+sx}{d_0+sy}\in\SL(2, \ZZ)$. We wish to choose $s$ such that:
\begin{itemize}
	\item $c=c_0+sx$ and $d=d_0+sy$ are nonnegative;
	\item $d$ is coprime to all numbers of the form $u_0+4yk$ where $0\leq k \leq \frac{n}{4xy}-\frac{u_0}{4y}$;
    \item if $4\nmid y$, then $d\equiv \opart{n}\pmod{4}$
\end{itemize}

We justify why this is possible. As $x,y>0$, the first condition is satisfied for any sufficiently large $s$. The second can be satisfied since $\gcd(d_0, y)=1$, making $d_0+sy$ an arithmetic progression with coprime residue and modulus. If $y$ is odd, then the third point can be satisfied simultaneously. 
If $2\mid\mid y$, then $d_0$ is necessarily odd from \eqref{eqn:xd0yc0den}, and the third condition holds for either all $s$ odd or all $s$ even.

\textbf{Applying Proposition~\ref{prop:mobiuskron}.}
Thus suppose we have chosen $s$, and hence $c$ and $d$, as desired.  Apply Proposition \ref{prop:mobiuskron} to the matrix $\sm{x}{y}{c}{d}$ and pair $(u, v)=(u_0+4yk,v_0-4xk)$ (note that the second chosen condition guarantees the hypotheses of Proposition~\ref{prop:mobiuskron} are satisfied), to obtain
\[
\kron{n}{cu_0+dv_0-4k}=(-1)^{\alpha}\mu_1\mu_2\kron{c}{d}\kron{u}{v}.
\]
Applying quadratic reciprocity and rearranging yields
\begin{equation}\label{eqn:kronjobgeneralden}
\kron{u}{v} = (-1)^{AB+AC+BC+AD+CE}\mu_1\mu_2\kron{c}{d}\kron{cu_0+dv_0-4k}{n},
\end{equation}
where
\[
A=\frac{\opart{u}-1}{2},\quad B=\frac{\opart{d}-1}{2},\quad C=\frac{\opart{(cu_0+dv_0-4k)}-1}{2},\quad D=\frac{\opart{v}-1}{2}, \quad E=\frac{\opart{n}-1}{2},
\]
and 
\[
\mu_1 = \begin{cases} 
	\kron{cduv+1}{2} & \text{if $2\mid\mid d$ or $2\mid\mid u$;}\\
	1 & \text{else;}
\end{cases}
,\quad
\mu_2 = \begin{cases}
	\kron{yu(cu+dv)+1}{2} & \text{if $2\mid\mid cu+dv$;}\\
	1 & \text{else.}
\end{cases}
\]

Analyzing the various imposed conditions allows us to greatly simplify \eqref{eqn:kronjobgeneralden}. First, $cu_0+dv_0-4k\equiv d\pmod{4}$, which is necessarily odd. Thus $B=C$.  Since $D\equiv 0\pmod{2}$, we obtain
\[AB+AC+BC+AD+CE\equiv B(B+E)\pmod{2}.\]
This is $0$ if and only if either $\opart{d}\equiv 1\pmod{4}$ or $\opart{d}\equiv\opart{n}\pmod{4}$. 
Also, $B=E$ if in addition $4\nmid y$ (using the choice of $c,d$).
Therefore the power of $-1$ in  \eqref{eqn:kronjobgeneralden} is $-1$ if and only if $4\mid y$, $x\equiv 3\pmod{4}$, and $\opart{n}\equiv 1\pmod{4}$ (using $xd\equiv 1\pmod{4}$ since $4\mid y$).

As $cu+dv\equiv d\equiv 1\pmod{2}$ ($d$ was chosen to be odd) and $4\mid u$, $\mu_1=\mu_2=1$. Putting this all together gives
\begin{equation}\label{eqn:kronjobbetterden}
\kron{u}{v} = \pm\kron{c}{d}\kron{cu_0+dv_0-4k}{n},
\end{equation}
where the $\pm$ is $-$ if and only if $4\mid y$, $x\equiv 3\pmod{4}$, and $\opart{n}\equiv 1\pmod{4}$.

\textbf{Finding the needed Kronecker symbol.}  The equation \eqref{eqn:kronjobbetterden} is a formula for $\kron{u}{v}$ in terms of $k$.  We now proceed to demonstrate that it takes a value $+1$ for some $k$ in the range \eqref{eqn:krange}.

\textbf{Case A:}  Assume the first condition of the lemma holds.  Then $n=2m^2=xu+yv$. Thus
\begin{equation}
    \label{eqn:ref2lemma}
\kron{u}{v} = \pm\kron{c}{d}\kron{cu_0+dv_0-4k}{n}=\pm\kron{c}{d}\kron{cu_0+dv_0-4k}{2}.
\end{equation}
The values of $k=0$ and $k=1$ give opposite values for this expression, so we are done as long as both are in our range, i.e. $1\leq\frac{n}{4xy}-\frac{u_0}{4y}$. This is implied by the condition of $n\geq 8xy$.

\textbf{Case B:}  Assume the second condition of the lemma holds.  Let $k=2k'$ be even, and consider
\[
\kron{u}{v}=\pm\kron{cu_0+dv_0-8k'}{2^{v_2(n)}}\kron{cu_0+dv_0-8k'}{\opart{n}}=\pm\kron{(cu_0+dv_0)/8-k'}{\opart{n}},
\]
where values independent of $k'$ (such as $\kron{c}{d}$) get absorbed into the $\pm$, and $(cu_0+dv_0)/8$ is regarded as any integer equivalent to this modulo $\opart{n}$. The possible values of $k'$ are
\[0\leq k'\leq \frac{n}{8xy}-\frac{u_0}{8y},\]
giving a sequence of at least $\lfloor\frac{n}{8xy}\rfloor$ consecutive integers. By the assumption, the resulting Kronecker symbol can be both $+1$ and $-1$ in this range, proving that $\kron{u}{v}=1$ has a solution.
\end{proof}

The above lemma does not apply when $n$ is a square. For this, we follow a similar proof, but the final analysis is slightly different.

\begin{lemma}\label{lem:denomsquareseventual}
	Let $x$ and $y$ be coprime positive integers, and let $n\geq 4xy$ be an integral square satisfying $n\equiv y\pmod{4}$. Furthermore, assume that one of the following is true:
     \begin{itemize}
         \item $y\equiv 1\pmod{4}$ and $\kron{x}{y}=1$;
         \item $(x, y)\equiv (1, 0)\pmod{4}$ and $\kron{x}{y}=1$;
         \item $(x, y)\equiv (3, 0)\pmod{4}$ and $\kron{x}{y}=\kron{-1}{y}$;
     \end{itemize}
    Then $n$ is a denominator in the orbit $\Psi\smcol{x}{y}$.
\end{lemma}

\begin{proof}
The proof of Lemma \ref{lem:denomeventual} up to equation \eqref{eqn:kronjobbetterden} is still valid. To summarize, we have a solution $u_0x+v_0y=n$ with $u_0\equiv 0\pmod{4}$ and $v_0\equiv 1\pmod{4}$, where $u_0\geq 0$ is minimized. We are looking for pairs $(u, v)=(u_0+4yk, v_0-4xk)$ with $0\leq k\leq\frac{n}{4xy}-\frac{u_0}{4y}$ and $\kron{u}{v}=1$.

Nonnegative integers $c, d$ were chosen such that $xd-yc=1$, $d$ is coprime to all possible $u$ values, and if $y$ is odd, then $d\equiv \opart{n}\equiv 1\pmod{4}$ (since $\opart{n}$ is still a square). As $n$ is a square,  \eqref{eqn:kronjobbetterden} implies that
\[\kron{u}{v}=\pm\kron{c}{d},\]
where the $\pm$ is $+$ except for $(x, y)\equiv (3, 0)\pmod{4}$.

\textbf{Case 1:}  Suppose $y$ is odd or $(x, y)\equiv (1, 0)\pmod{4}$, and suppose $\kron{x}{y}=1$.  Then $d\equiv 1\pmod{4}$, and Lemma \ref{lem:allkronequal} applies to $\sm{x}{c}{y}{d}$, showing that $\kron{c}{d}=\kron{x}{y}$. Therefore $\kron{u}{v}=1$, so there is a solution as long as $\frac{n}{4xy}-\frac{u_0}{4y}\geq 0$, which follows from $n\geq 4xy$.

\textbf{Case 2:}  Suppose that $(x, y)\equiv (3, 0)\pmod{4}$ and $\kron{x}{y}=\kron{-1}{y}$.  Then $\kron{u}{v}=-\kron{c}{d}$. We have $d\equiv 3\pmod{4}$, and as in the proof of Lemma \ref{lem:allkronequal}, 
\[
\kron{c}{d}=\kron{y}{d}\kron{cy}{d}=(-1)^{(\opart{y}-1)/2}\kron{d}{y}\kron{dx-1}{d}=-\kron{-1}{y}\kron{x}{y}\kron{xd}{y}=-\kron{cy+1}{y}=-1.
\]
Thus $\kron{u}{v}=1$, and we finish as in the first case.
\end{proof}

The extra difficulty for the numerator case is that $v\pmod{4}$ is not necessarily fixed; otherwise the proof follows a similar path.

\begin{lemma}\label{lem:numeventual}
	Let $x$ and $y$ be coprime positive integers, and let $n$ be a positive integer satisfying $n\equiv x\pmod{\gcd(y, 4)}$. Furthermore, assume that at least one of the following conditions holds:
     \begin{itemize}
         \item $n\geq 8xy$ and $n=2m^2$ for some integer $m$;
         \item in any consecutive sequence of $\lfloor\frac{n}{8xy}\rfloor$ integers, there is an integer $r_-$ with $\kron{r_-}{\opart{n}}=-1$ and an integer $r_+$ with $\kron{r_+}{\opart{n}}=1$.
     \end{itemize}
    Then $n$ is a numerator in the orbit $\Psi\smcol{x}{y}$.
\end{lemma}

\begin{proof}The proof follows the same plan as for Lemma~\ref{lem:denomeventual}, so we will highlight the differences.

\textbf{Setting up the list of solutions to $n=ux+vy$ with $u \equiv 1 \pmod{4}$.}
As $n\equiv x\pmod{\gcd(y, 4)}$, there exists integers $(u, v)$ satisfying $ux+vy=n$ with $u \equiv 1 \pmod 4$. Pick the solution with smallest nonnegative value of $u$, say $(u_0, v_0)$, and let $h\in\{1, 2, 4\}$ be the smallest positive integer such that $4\mid hy$. Then $(u, v) = (u_0+hyk, v_0-hxk)$ for $k\in\ZZ$ gives the general solution to the system of equations $ux+vy=n$ and $u \equiv 1 \pmod 4$.

In particular,
\begin{itemize}
	\item $0\leq u_0 < hy$;
	\item this solution is a nonnegative pair if and only if $0\leq k \leq \frac{n}{hxy}-\frac{u_0}{hy}$.
\end{itemize}
Lemma~\ref{lem:repcond} implies that $n$ is a numerator of the orbit $\Psi\smcol{x}{y}$ if and only if there exists an integer $k$ with
\[0\leq k \leq \frac{n}{hxy}-\frac{u_0}{hy}\quad\text{and}\quad\kron{u}{v}=\kron{u_0+hyk}{v_0-hxk}=1.\]

\textbf{Finding a suitable matrix for Proposition~\ref{prop:mobiuskron}.}
As before, fix a pair of integers $(c_0, d_0)$ such that 
\begin{equation*}
    xd_0-yc_0=1. 
\end{equation*}
For any integer $s$, we have $M=\sm{x}{y}{c_0+sx}{d_0+sy}\in\SL(2, \ZZ)$. We wish to choose $s$ such that:
\begin{itemize}
	\item $c=c_0+sx$ and $d=d_0+sy$ are nonnegative;
	\item $d$ is coprime to all numbers of the form $u_0+hyk$ where $0\leq k \leq \frac{n}{hxy}-\frac{u_0}{hy}$;
    \item $d\equiv \opart{n}\pmod{4}$ (in particular, $d$ is odd).
\end{itemize}
The justification for the existence of $s$ is exactly as in Lemma~\ref{lem:denomeventual}, except that we must also demonstrate the last bullet when $4\mid y$.  In this case, the assumptions on $x$ imply that $x$ is odd and $x\equiv n\pmod{4}$. 
 Furthermore, $xd_0\equiv xd_0-yc_0\equiv 1\pmod{4}$, whence $\opart{n}\equiv n\equiv x \equiv d_0\equiv d\pmod{4}$.

\textbf{Applying Proposition~\ref{prop:mobiuskron}.}
As in Lemma~\ref{lem:denomeventual}, suppose we have chosen $s$, and hence $c$ and $d$, as desired, and apply Proposition \ref{prop:mobiuskron} to the matrix $\sm{x}{y}{c}{d}$ and pair $(u, v)=(u_0+hyk,v_0-hxk)$, to obtain
\begin{equation}\label{eqn:kronjobgeneral}
\kron{u}{v} = (-1)^{AB+AC+BC+AD+CE}\mu_1\mu_2\kron{c}{d}\kron{cu_0+dv_0-hk}{n},
\end{equation}
where
\[
A=\frac{\opart{u}-1}{2},\quad B=\frac{\opart{d}-1}{2},\quad C=\frac{\opart{(cu_0+dv_0-hk)}-1}{2},\quad D=\frac{\opart{v}-1}{2}, \quad E=\frac{\opart{n}-1}{2},
\]
and 
\[
\mu_1 = \begin{cases} 
	\kron{cduv+1}{2} & \text{if $2\mid\mid d$ or $2\mid\mid u$;}\\
	1 & \text{else;}
\end{cases}
,\quad
\mu_2 = \begin{cases}
	\kron{yu(cu+dv)+1}{2} & \text{if $2\mid\mid cu+dv$;}\\
	1 & \text{else.}
\end{cases}
\]
As $u=u_0+hyk\equiv 1\pmod{4}$, we have that $A$ is even, so $AB+AC+BC+AD+CE\equiv C(B+E)\pmod{2}$. Since $\opart{d}\equiv d\equiv\opart{n}\pmod{4}$ by assumption, we have $B\equiv E\pmod{2}$, so $C(B+E)\equiv 0\pmod{2}$. In particular, the power of $-1$ in \eqref{eqn:kronjobgeneral} is always $1$.

Next, $d$ is odd and $u\equiv 1\pmod{4}$, hence $\mu_1=1$. This gives the much more palatable
\begin{equation}\label{eqn:kronjobbetter}
\kron{u}{v} = \mu_2\kron{c}{d}\kron{cu_0+dv_0-hk}{n}.
\end{equation}

\textbf{Finding the needed Kronecker symbol.}  We break into cases.

\textbf{Case A:}  Assume the first condition of the lemma holds.  Then $2m^2=n = xu+yv$. Since $\gcd(cu+dv,2m^2) = \gcd(cu+dv,xu+yv)=\gcd(u, v)$ is odd (since $u$ is odd), we conclude $cu+dv$ is odd, and therefore $\mu_2=1$. Furthermore, $y$ is necessarily odd, so $h=4$, and \eqref{eqn:ref2lemma} from Lemma~\ref{lem:denomeventual} holds, where the sign is positive.
As in Lemma~\ref{lem:denomeventual}, we are done if $n \geq 8xy$.

\textbf{Case B:} We will assume instead that the second condition of the lemma holds.  We  will divide into cases based on $v_2(y)$.

\textbf{Case B1:} $4\mid y$.  Hence $h=1$ and $n$ is odd (from $n\equiv x\pmod{4}$ and $x$ being odd). If $2\mid\mid cu+dv$, then $8\mid yu(cu+dv)$, hence $\mu_2=1$ always. Thus
\[
\kron{u}{v} = \kron{c}{d}\kron{cu_0+dv_0-k}{n}.
\]
As $u_0 < y$, we have a block of at least $\lfloor\frac{n}{xy}\rfloor$ consecutive integers that we are evaluating in $\kron{\cdot}{n}$, and we will be done if they take the values of $-1$ and $1$ (since one of those will equal the fixed quantity $\kron{c}{d}$). This is implied by our assumption.

\textbf{Case B2:} $2\mid\mid y$.  Hence $h=2$ and $n$ is odd. Then $cu+dv=cu_0+dv_0-2k$ is always even or always odd. If it is always odd, then $\mu_2=1$, and we obtain
\[
\kron{u}{v} = \kron{c}{d}\kron{cu_0+dv_0-2k}{n}=\kron{c}{d}\kron{2}{n}\kron{(cu_0+dv_0)/2-k}{n},
\]
where division is taken modulo the denominator of the Kronecker symbol, i.e. modulo $n$.  Similarly to Case B1, we have a sequence of at least $\lfloor\frac{n}{2xy}\rfloor$ consecutive Kronecker symbols, so we can pick a value of $k$ giving $\kron{u}{v}=1$ (as $\kron{c}{d}\kron{2}{n}$ is independent of $k$).

If $cu+dv$ is always even, then by considering only even $k$ or only odd $k$, we can ensure that $4\mid cu+dv$, once again making $\mu_2=1$. Analogous equations hold, yielding a sequence of at least $\lfloor\frac{n}{4xy}\rfloor$ consecutive Kronecker symbols, and we are done.

\textbf{Case B3:} $y$ is odd. Hence $h=4$. As $cu+dv=cu_0+dv_0-4k$, whether or not $2\mid\mid cu+dv$ is independent of $k$. If $2\mid\mid cu+dv$, then we see that incrementing $k$ by 1 changes $yu(cu+dv)+1$ by $4yu\equiv 4\pmod{8}$, i.e. $\mu_2$ swaps sign. In particular, taking either all even $k$ or all odd $k$ causes $\mu_2$ to be constant as a function of $k$.  Picking $k=2k'$ to be even and applying the same trick as in Case B of Lemma~\ref{lem:denomeventual} gives 
\[
\kron{u}{v}=\mu_2\kron{c}{d}\kron{cu_0+dv_0}{2^{v_2(n)}}\kron{(cu_0+dv_0)/8-k'}{\opart{n}}.
\]
As we have at least $\lfloor\frac{n}{8xy}\rfloor$ consecutive values of $k'$, we are done.

\end{proof}

Again, we consider square numerators separately.

\begin{lemma}\label{lem:numsquareseventual}
	Let $x$ and $y$ be coprime positive integers, and let $n\geq 8xy$ be an integral square satisfying $n\equiv x\pmod{\gcd(y, 4)}$. Furthermore, assume that one of the following is true:
     \begin{itemize}
         \item $y$ is odd and $\kron{x}{y}=1$;
         \item $(x, y)\equiv (1, 0)\pmod{4}$ and $\kron{x}{y}=1$;
         \item $y\equiv 2\pmod{4}$;
     \end{itemize}
    Then $n$ is a numerator in the orbit $\Psi\smcol{x}{y}$.
\end{lemma}

\begin{proof}
The proof of Lemma \ref{lem:numeventual} up to equation \eqref{eqn:kronjobbetter} is still valid. To summarize, we have a solution $u_0x+v_0y=n$ with $u_0\equiv 1\pmod{4}$, where $u_0>0$ is minimized. We are looking for pairs $(u, v)=(u_0+hyk, v_0-hxk)$ with $0\leq k\leq\frac{n}{hxy}-\frac{u_0}{hy}$ with $\kron{u}{v}=1$, where $h\in\{1, 2, 4\}$ is the smallest integer satisfying $4\mid hy$.

Nonnegative integers $c, d$ were chosen such that $d\equiv \opart{n}\equiv 1\pmod{4}$ (since $\opart{n}$ is still a square), $xd-yc=1$, and $d$ is coprime to all possible $u$ values. Since $n$ is a square, \eqref{eqn:kronjobbetter} implies that
\[\kron{u}{v}=\mu_2\kron{c}{d},\quad\text{where }\mu_2 = \begin{cases}
	\kron{yu(cu+dv)+1}{2} & \text{if $2\mid\mid cu+dv$;}\\
	1 & \text{else.}
\end{cases}\]

\textbf{Case 1:}  $y\not\equiv 2\pmod{4}$.  Then by $d \equiv 1 \pmod{4}$, the proof of Lemma \ref{lem:allkronequal} applies to $\sm{x}{y}{c}{d}$, so $\kron{u}{v} = \mu_2\kron{c}{d}=\mu_2\kron{x}{y} = \mu_2$.

\textbf{Case 1A:} $4\mid y$. 
As in the proof of Lemma \ref{lem:numeventual}, $\mu_2=1$, and we are done.

\textbf{Case 1B:} $y$ is odd.  If $n$ is even, then  $\gcd(cu+dv,n) = \gcd(cu+dv,xu+yv)=\gcd(u, v)$ is odd (as in the proof of Lemma~\ref{lem:numeventual}, Case A), hence we conclude $cu+dv$ is odd, and therefore $\mu_2=1$. Otherwise, $n\equiv 1\pmod{4}$, and
\[1=xd-yc\equiv x-yc\pmod{4},\quad 1\equiv n=ux+vy\equiv x+vy\pmod{4}.\]
This combines to $yc+vy\equiv 0\pmod{4}$, so $4\mid c+v$, and thus $cu+dv\equiv c+v\equiv 0\pmod{4}$, once again implying that $\mu_2=1$. 

In both cases 1A and 1B we are done, because $n \geq xy$, so the range of valid $k$ is non-empty.

\textbf{Case 2:} $2\mid\mid y$.  Hence $h=2$, and $n,x$ are odd. As $n$ is a square, $n\equiv 1\pmod{4}$, so $1\equiv n\equiv ux+vy\equiv x+2v\pmod{4}$. Similarly, $xd-yc=1$, so $x+2c\equiv 1\pmod{4}$, implying that $v\equiv c\pmod{2}$. In particular, $cu+dv=cu_0+dv_0-2k$ is always even. Considering $k=0, 1, 2, 3$, exactly two of these values will trigger $2\mid\mid cu+dv$, and their values of $\mu_2$ will be opposite. In particular, one of them will give $\kron{u}{v}=\mu_2\kron{c}{d}=1$, proving that $n$ is in the orbit. This will be valid as long as $\frac{n}{hxy}-\frac{u_0}{hy}\geq 3$, which follows from $n\geq 8xy$.
\end{proof}

\section{Proof of Theorems \ref{thm:semigroupreciprocity} and \ref{thm:TableForPsi}}
\label{sec:psi-few}

Theorem \ref{thm:semigroupreciprocity} follows immediately from Lemma~\ref{lem:kronpreserved}.

\begin{proof}[Proof of Theorem~\ref{thm:semigroupreciprocity}]
	First, assume one of the first two conditions. By Lemma~\ref{lem:kronpreserved}, if $\smcol{u}{v}$ is in the orbit, then $\kron{u}{v}=-1$, hence neither $u$ nor $v$ can be a square.
	
	Now assume that $(x, y)\equiv (3, 0)\pmod{4}$. Since $ax+by\equiv 3\pmod{4}$, it can never be a square, so it suffices to check if $cx+dy$ can be a square. If it is, then $\opart{(cx+dy)}\equiv 1\pmod{4}$, whence by Lemma~\ref{lem:kronpreserved},
	\[\kron{ax+by}{cx+dy}=(-1)^{(\opart{y}-1)/2}\kron{x}{y}=\kron{-1}{y}\kron{x}{y}=-1,\]
	again a contradiction.
\end{proof}

This also gives the positive result for the reciprocity obstructions listed in Theorem \ref{thm:TableForPsi}. The necessity of the congruence conditions follows immediately from Lemma \ref{lem:repcond}. In order to prove that this completely describes the orbit, we invoke the lemmas from Section \ref{sec:psi-many} alongside an analytic input from Section \ref{sec:analytic}.

\subsection{Analytic lemma}\label{sec:analytic}

For a positive non-square integer $n$, let $f(n)$ denote the smallest positive integer $m$ such that in any sequence of $m$ consecutive integers, there exists at least one value $x$ for which $\kron{x}{n}=1$, and one value $y$ for which $\kron{y}{n}=-1$. That is,
\[
f(n) 
= \min \left\{ m : \forall a \in \ZZ, \{-1, 1\} \subseteq \left\{ \kron{x}{n} : a < x \leq a+m \right\} \right\}.
\]
We will require two upper bounds for $f(n)$. The first one is a quick modification of the trivial bound $f(n)\leq n$.

\begin{lemma}\label{lem:fntrivialupperbound}
    Let $n\geq 5$ be odd and not square. Then $f(n)\leq n-1$.
\end{lemma}
\begin{proof}
    Let $\phi(n)$ denote Euler's totient function, which denotes the number of positive integers up to $n$ that are coprime to $n$. As $n$ is not a square, in the sequence $\kron{0}{n}, \kron{1}{n}, \ldots, \kron{n-1}{n}$, there are exactly $\phi(n)/2$ values of $1$, and the same number of $-1$'s. Thus if $\phi(n)\geq 4$, then any sequence of $n-1$ numbers still has one with Kronecker being $-1$, and one with Kronecker being $1$. If $n$ has a prime factor at least $5$ then $\phi(n)\geq 4$, and if $9\mid n$ this this again holds. All odd $n\geq 5$ satisfy at least one of these conditions, giving the result.
\end{proof}

For large $n$, we require a better bound. If $n\geq 2.5\cdot 10^9$ is prime, then Trevi\~{n}o shows that $f(n)\leq 1.55 n^{1/4} \log{n}$ \cite[Theorem 1]{Trevino12}. Since we require a result for non-primes as well, we need a new argument. Lemma~\ref{lem:fnupperbound} is good enough for our purposes, but presumably far from optimal (this function can be computationally explored with \texttt{kron\_seq\_both} in \cite{GHSemigroup}).

\begin{lemma}\label{lem:fnupperbound}
    Let $n\geq 3.41\cdot 10^6$ be a non-square integer with $n\not\equiv 2\pmod{4}$. Then
    \[f(n) \leq 0.942836\sqrt{n}\log{n}\log\log{n}.\]
\end{lemma}

\begin{proof}[Proof of Lemma~\ref{lem:fnupperbound}]
To begin, we will assume $n\geq 3.41\cdot 10^6$ everywhere. Define the non-principal Dirichlet character of modulus $n$, $\chi_n(x):=\kron{x}{n}$.

An explicit upper bound for the P\'{o}lya-Vinogradov inequality given by Frolenkov and Soundararajan \cite[Corollary 1]{FrolenkovSound} states that for $q \ge 100$, and $\chi$ a non-principal character of modulus $q$,
\[
S_{M,N}(\chi):=
\left|\sum_{M< x\leq M+N}\chi(x)\right|\leq
\frac{1}{\pi \sqrt{2}} \sqrt{q} \log q + \left(1 + \frac{6}{\pi \sqrt{2}}\right) \sqrt{q}.
\]

Under our assumptions that $q=n$ and $n\geq 3.41\cdot 10^6$, we obtain
\begin{equation}\label{eqn:SMNupper}
S_{M, N}(\chi_n)\leq \left(\frac{1}{\pi\sqrt{2}}+\frac{1}{\log{n}}\left(1 + \frac{6}{\pi \sqrt{2}}\right)\right)\sqrt{n}\log{n}\leq 
c_1 \sqrt{n}\log{n},
\end{equation}
where $c_1\approx 0.381338$.

Let $N = f(n) - 1$, and consider a sequence of $N$ consecutive integers $M+1, \ldots, M+N$ for which $\kron{x}{n}\neq 1$ (or for which $\kron{x}{n}\neq -1$). Since $\kron{x}{n}=\chi_n(x)$, there is no cancellation in $S_{M, N}(\chi_n)$, hence
\begin{equation}\label{eqn:SMNdefaultlower}
S_{M, N}(\chi_n)\geq \#\{x:M< x\leq M+N, \gcd(x, n)=1\}.
\end{equation}

To bound this, let $g_n(A)$ denote the number of integers from $1$ to $A$ that are coprime to $n$. A classic formula (for example, see the proof of Theorem 8.29 of \cite{NZMIntroNT}) gives that
\[g_n(A) = \sum_{d\mid n}\mu(d)\left\lfloor\frac{A}{d}\right\rfloor = A\frac{\phi(n)}{n}-\sum_{d\mid n}\mu(d)\left\{\frac{A}{d}\right\},\]
where $\mu(d)$ is the M\"{o}bius function and $\{x\}=x-\lfloor x\rfloor$. In particular, we have
\[\left|g_n(A)-A\frac{\phi(n)}{n}\right|<\sum_{d\mid n}|\mu(d)|\leq 2^{\omega(n)},\]
where $\omega(n)$ is the number of distinct prime factors of $n$.
Combining this with \eqref{eqn:SMNdefaultlower} yields
\[
S_{M, N}(\chi_n)\geq g_n(M+N)-g_n(M) \geq N\frac{\phi(n)}{n}-2^{\omega(n) + 1}.
\]
Using the upper bound \eqref{eqn:SMNupper} for $S_{M, N}$ and $f(n)=N+1$ gives an upper bound of
\begin{equation}
    \label{eqn:fbound}
f(n) \leq \frac{n}{\phi(n)}\left( \left(c_1+ \frac{ \phi(n) }{ n^{3/2} \log n} \right)\sqrt{n}\log{n}+2^{\omega(n)+1} \right) \leq
\frac{n}{\phi(n)}\left(0.381374\sqrt{n}\log{n}+2^{\omega(n)+1} \right).
\end{equation}
A bound for $\frac{n}{\phi(n)}$ can be found in Theorem 15 of \cite{RS62}:
\[\frac{n}{\phi(n)} < e^{\gamma}\log\log{n}+ \frac{2.50637}{\log\log{n}},\]
where $\gamma\approx 0.577216$ is the Euler–Mascheroni constant. Using $n\geq 3.41\cdot 10^6$ yields
\begin{equation}
    \label{eqn:phibound}
\frac{n}{\phi(n)}\leq 2.122133\log\log{n}.
\end{equation}

The final ingredient is a bound of Robin (Th\'{e}or\`{e}me 13 of \cite{Robin83}), which gives
\begin{equation*}
\omega(n)\leq \frac{\log{n}}{\log\log{n}-1.1714}.
\end{equation*}
Therefore
\[2^{\omega(n)}\leq n^{\log{2}/(\log\log{n}-1.1714)}\leq n^{0.450254}.\]
This gives
\begin{equation}
    \label{eqn:omegabound}
0.381374\sqrt{n}\log{n}+2^{\omega(n)+1}
\leq
\sqrt{n}\log{n}\left(0.381374+\frac{2}{n^{0.049746}\log{n}}\right)\leq
0.4442869 \sqrt{n}\log{n}.
\end{equation}

Combining \eqref{eqn:fbound}, \eqref{eqn:phibound} and \eqref{eqn:omegabound} gives
\[f(n) \leq 0.942836\sqrt{n}\log{n}\log\log{n},\]
proving Lemma~\ref{lem:fnupperbound}.
\end{proof}

\begin{remark}
    The only change to covering the $n\equiv 2\pmod{4}$ case would be that $\chi_n(x)=\kron{x}{n}$ now has modulus $4n$ instead of $n$, giving slightly worse bounds.
\end{remark}

\subsection{Proof of Theorem \ref{thm:TableForPsi}}
\label{sec:proof26}
As previously remarked, the necessity of the congruence conditions has been shown, and the existence of the claimed reciprocity obstructions proven. It remains to show that all other sufficiently large integers appear. The bulk of the work will come in applying Lemmas \ref{lem:denomeventual}, \ref{lem:numeventual}, where for $n\neq m^2, 2m^2$, it must be shown that $f(\opart{n})\leq \left\lfloor\frac{n}{8xy}\right\rfloor$.

\begin{proof}[Proof of Theorem \ref{thm:TableForPsi}]
If $x=0$ or $y=0$, the result follows immediately from Corollary~\ref{cor:psiorbitwithzero}. Otherwise, assume that $x,y>0$, and assume that $n$ is a non-square satisfying the congruence obstructions from Table~\ref{table:psi1obstructions}.

Next, also assume that $n\geq 2^kr$ and $n\geq 2^k\cdot 3.41\cdot 10^6$, where
\[\frac{\sqrt{r}}{\log{r}\log{\log{r}}}= 7.542795 xy,\]
and $k=\lfloor\log_2 8xy\rfloor$. First, we show that $r\geq 30000xy$. Let $xy=a$ and assume $r<30000a=:Ca$. Thus
\[
1=7.542795\frac{a\log{r}\log\log{r}}{\sqrt{r}}>7.542795\frac{a\log{Ca}\log\log{Ca}}{\sqrt{Ca}}\geq 7.542795\frac{\log{C}\log\log{C}}{\sqrt{C}}>1.04,
\]
a contradiction. Therefore $n\geq 8r\geq 240000xy$.

If $n$ is twice a square, then we are done by Lemmas \ref{lem:denomeventual}, \ref{lem:numeventual}. Otherwise, write $n=2^w\opart{n}$, and assume first that $2^w\geq 8xy$. In this case, 
\[\left\lfloor\frac{n}{8xy}\right\rfloor\geq\frac{n}{8xy}-1\geq\opart{n}-1\geq f(\opart{n})\]
by Lemma \ref{lem:fntrivialupperbound}, assuming $\opart{n}\geq 5$. If $\opart{n}=3$, then $n\geq 240000xy$ implies
\[\left\lfloor\frac{n}{8xy}\right\rfloor\geq 29999\geq 3 = f(\opart{n}).\]
The result now follows from Lemmas \ref{lem:denomeventual}, \ref{lem:numeventual}.

If $2^w<8xy$, then $w\leq k$, hence $\opart{n}\geq r$ and $\opart{n}\geq 3.41\cdot 10^6$. Thus
\[
\left\lfloor\frac{n}{8xy}\right\rfloor\geq
\frac{\opart{n}}{8xy}-1\geq
0.942836\sqrt{\opart{n}}\log{\opart{n}}\log{\log{\opart{n}}}
\geq f(\opart{n}),
\]
where the second inequality follows from both lower bounds on $\opart{n}$ in the previous sentence, the definition of $r$, and the fact that $\sqrt{r}/\log r \log\log r$ is an increasing function, and the third inequality is Lemma~\ref{lem:fnupperbound}. Lemmas \ref{lem:denomeventual} and \ref{lem:numeventual} again give the result.

Finally, assume $n\geq 8xy$ is a square. Lemmas \ref{lem:denomsquareseventual} and \ref{lem:numsquareseventual} fill in all cases where a reciprocity obstruction was not claimed.
\end{proof}

\section{An explicit orbit}\label{sec:explicit}

Since Theorem~\ref{thm:TableForPsi} is completely explicit, combining it with a computation allows us to completely describe given orbits of $\Psi$. 
Theorem~\ref{thm:psi123orbit} provides the first example of a completely described orbit with reciprocity obstructions.

\begin{proof}[Proof of Theorem \ref{thm:psi123orbit}]
    Using the notation of Theorem~\ref{thm:TableForPsi}, we have $8xy=48$, hence $2^k=32$. From $\frac{\sqrt{r}}{\log{r}\log{\log{r}}}=45.25677$, we derive $r\geq 3.41\cdot 10^6$. Therefore if $n\geq 1.0912\cdot 10^8$ is a non-square, then $n$ is a numerator in the orbit $\Psi\smcol{2}{3}$. We have a reciprocity obstruction, so squares do not occur, and thus it suffices to compute the non-squares in the orbit up to $1.0912\cdot 10^8$. This can be done with Lemma \ref{lem:repcond}. Code to do this in under half a minute is found in \cite{GHSemigroup}, in the method \texttt{test\_psi\_23orbit}.
\end{proof}

\section{Subsemigroups of $\Psi$ of Hausdorff dimension less than $1$}\label{sec:thin}

\subsection{Limit sets and Hausdorff dimensions for continued fraction semigroups}
\label{sec:hausdorff}

Recall that $\gamma = \genmtx \in \SL(2, \ZZ)$ acts on $\PP^1(\RR)$ by M\"{o}bius maps, i.e.  $\gamma(x) = \frac{ax+b}{cx+d}$.
Associated to any continued fraction semigroup $\Gamma \subseteq \SL(2,\ZZ)^{\ge 0}$ is its limit set $\Lambda(\Gamma)$:
\[
\Lambda(\Gamma) := \left\{
\lim_{n \rightarrow \infty} \gamma_0  \cdots  \gamma_n(1) : \gamma_j \in \Gamma \setminus \{ I \}
\right\}\subseteq (0, \infty).
\]
Equivalently, one can allow $\gamma_j$ from a generating set for $\Gamma$ in the expression above. The Hausdorff dimension of this limit set will be denoted $\delta(\Gamma)$.

Our first goal is to show that the limit set of $\Psi$ has Hausdorff dimension $1$, which we approach by starting with $\SL(2, \ZZ)^{\geq 0}$, and then $\Gamma_1(4)^{\geq 0}$.

\begin{lemma}\label{lem:sl2zgeq0hdim}
    $\delta(\SL(2, \ZZ)^{\geq 0}) = 1$.
\end{lemma}
\begin{proof}
    Let $\alpha=[a_0;a_1,a_2,\ldots]$ be a positive irrational number, and recall that elements of $\SL(2, \ZZ)^{\geq 0}$ can be expressed uniquely as words in $L,R$ (with nonnegative powers). By \eqref{eqn:evencontfrac} in Section 2,
    \[
    \frac{x}{y}=[a_0;a_1,a_2,\ldots,a_{2n-1}]\quad\Leftrightarrow\quad \frac{x}{y}=L^{a_0}R^{a_1}\cdots L^{a_{2n-2}}R^{a_{2n-1}-1}(1),
    \]
    whence $\alpha\in\Lambda(\SL(2, \ZZ)^{\geq 0})$, and therefore $\delta(\SL(2, \ZZ)^{\geq 0}) = 1$.
\end{proof}

Given two elements $\gamma_1,\gamma_2\in\SL(2, \ZZ)^{\geq 0}$, we say that $\gamma_1$ is an \emph{initial subword} of $\gamma_2$ if $\gamma_1^{-1}\gamma_2\in\SL(2, \ZZ)^{\geq 0}$. Equivalently, $\gamma_1$ is an initial subword of $\gamma_2$ if and only if $\gamma_2$'s LR-word begins with $\gamma_1$'s LR-word.

We extend the definition of limit set to \emph{subsets} $\Gamma\subseteq \SL(2,\ZZ)^{\ge 0}$, namely, $\Lambda(\Gamma)$ is the collection of limits $\lim_{n \rightarrow \infty} \gamma_n(1)$ of all the chains of words $\gamma_0 \le \gamma_1 \le \cdots$ where $\gamma_i \in \Gamma$ and $ a \le b$ indicates that $a$ is an initial subword of $b$. Write $\delta(\Gamma) = \delta(\Lambda(\Gamma))$, as before.

\begin{lemma}
\label{lemma:coset}
    Let $\Gamma \subseteq \SL(2,\ZZ)^{\ge 0}$ be a subsemigroup.  Suppose that $\Gamma=\bigcup_{i}  \Gamma_i$ for some finite set of subsets $\Gamma_i \subseteq \Gamma$.  Suppose further that for all $i,j$ there exists $\gamma_{i,j} \in \SL(2,\ZZ)^{\ge 0}$ so that $\gamma_{i,j} \Gamma_i \subseteq \Gamma_j$.  Then $\delta(\Gamma_i) = \delta(\Gamma)$.
\end{lemma}

\begin{proof}
Let $\alpha \in \Lambda(\Gamma)$, so that $\alpha = \lim_{n \rightarrow \infty} \gamma_0 \gamma_1 \cdots \gamma_n (1)$.  Considering all subwords $\gamma_0 \gamma_1 \cdots \gamma_n$, there exists some $i$ for which infinitely many such subwords lie in $\Gamma_i$, hence $\alpha \in \Lambda(\Gamma_i)$.   Therefore
\[
\bigcup_{i} \Lambda(\Gamma_i)\subseteq \Lambda(\Gamma) \subseteq \bigcup_{i} \Lambda(\Gamma_i),
\]
and equality follows. The Hausdorff dimension of a union is the supremum of the Hausdorff dimensions, hence $\delta(\Gamma)=\sup_i\delta(\Gamma_i)$.
Furthermore, 
\[\gamma_{i,j}\gamma_{j,i}\Lambda(\Gamma_j)\subseteq\gamma_{i,j} \Lambda(\Gamma_i) \subseteq \Lambda(\Gamma_j)\]

Therefore all the $\delta(\Gamma_i)$ are equal (since Hausdorff dimension is preserved by M\"obius transformation), and equal to $\delta(\Gamma)$.
\end{proof}

\begin{proposition}
\label{prop:dim1}
$\delta(\Psi) = 1$.
\end{proposition}

\begin{proof}
First, we show that $\delta(\Gamma_1(4)^{\ge 0}) = 1$.  This uses Lemma~\ref{lemma:coset}, by observing that
\[
\SL(2,\ZZ)^{\ge 0} = \bigcup_{M \in \SL(2,\ZZ)/\Gamma_1(4)}  \Gamma_M,
\]
where $\Gamma_M = \{ \gamma \in \SL(2,\ZZ)^{\ge 0} : \gamma \equiv M \pmod{4} \}$.  Since $\delta(\SL(2,\ZZ)^{\ge 0})=1$ (by Lemma \ref{lem:sl2zgeq0hdim}) and $\Gamma_1(4)^{\ge 0}$ is a union of $\Gamma_M$'s, we are done.

Finally, $\delta(\Psi) = 1$ also follows from Lemma~\ref{lemma:coset}, by observing that
$\Gamma_1(4)^{\ge 0} = \Psi \cup \Psi'$, where $\Psi' = \Gamma_1(4)^{\ge 0} \setminus \Psi$.  If $\gamma \in \Psi'$, then $\gamma \Psi \subseteq \Psi'$ and $\gamma \Psi' \subseteq \Psi$ by \eqref{eqn:kronchange} in Section 3.  
\end{proof}

\subsection{Iterated function systems}

Given a matrix $\gamma=\genmtx\in\SL(2, \ZZ)^{\geq 0}$, associate to it
\[T_{\gamma}:=R\gamma R^{-1}=\lm{a-b}{b}{a+c-b-d}{b+d}.\]

Then $T_{\gamma}: [0,1] \rightarrow [0,1]$ is a conformal map, which is a contraction as long as $\gamma$ is not the identity matrix. Recall that $\SL(2, \ZZ)^{\geq 0}=\langle L, R\rangle^+$, and note
\[T_L(x) = \frac{1}{2-x}:[0, 1]\rightarrow [1/2, 1], \quad T_R(x) = \frac{x}{x+1}:[0, 1]\rightarrow [0, 1/2].\]

For a non-empty subset $\mathcal{W}\subseteq\SL(2, \ZZ)^{\geq 0}$ not containing the identity, consider the set
\[
\mathcal{C}_\mathcal{W} := \{ T_\gamma : \gamma \in \mathcal{W} \}.
\]

If $\mathcal{W}$ is finite, then two things happen:
\begin{enumerate}
    \item $\mathcal{C}_\mathcal{W}$ satisfies the \emph{open set condition}: 
  there exists an open set $U \subseteq [0,1]$ such that the images $T_\gamma(U)$ for $\gamma \in \mathcal{W}$ are pairwise disjoint. 
    \item $\mathcal{C}_\mathcal{W}$ satisfies the \emph{bounded distortion property}: there exists a $K_\mathcal{W} \ge 1$ such that for all $\gamma \in \mathcal{C}_\mathcal{W}$, and $x,y \in [0,1]$.
\[
|T_\gamma'(x)| \le K_\mathcal{W} |T_\gamma'(y)|.
\]
    \end{enumerate} 
Thus the set $\{ T_\gamma : \gamma \in \mathcal{W} \}$ forms a \emph{conformal iterated function system}.  For details, see \cite{MU}.

To any such finite $\mathcal{W}$, define the semigroups
\[\Gamma_\mathcal{W} := \left\langle \gamma : \gamma \in \mathcal{W} \right\rangle^+,\quad \Gamma_\mathcal{W}^R := \left\langle T_{\gamma} : \gamma \in \mathcal{W} \right\rangle^+ = R\Gamma_{\mathcal{W}}R^{-1}.\]

The limit set of a conformal iterated function system $\{ T_\gamma : \gamma \in \mathcal{W} \}$ is defined as
\begin{equation}
    \label{eqn:IFSlimitset}
\bigcup_{\substack{\gamma_0\gamma_1 \cdots \\ \gamma_i \in \mathcal{W}}}
\bigcap_{n=1}^\infty T_{\gamma_0 \cdots \gamma_n}([0,1]) = 
\bigcap_{n=1}^\infty \bigcup_{\gamma_0,\ldots, \gamma_n \in \mathcal{W}} T_{\gamma_0 \cdots \gamma_n}([0,1]),
\end{equation}
where the equality follows from the finiteness of the iterated function system.  This definition coincides with $\Lambda(\Gamma_\mathcal{W}^R)$ in our situation, as $\lim_{n \rightarrow \infty} T_{\gamma_0 \cdots \gamma_n}(1) = \cap_{n=1}^\infty T_{\gamma_0 \cdots \gamma_n}([0,1])$.

Since the Hausdorff dimensions of $\Lambda(\Gamma_\mathcal{W})$ and $\Lambda(\Gamma_\mathcal{W}^R)$ are equal, we can take advantage of the conformal iterated function system setup of $\Gamma_\mathcal{W}^R$ for computation.

\subsection{Hausdorff dimension of subsemigroups of $\Psi$ (proof of Theorem~\ref{thm:epsilon-dim})}
\label{sec:epsilon-dim}

In this section, we show that $\Psi$ has finitely generated subsemigroups whose Hausdorff dimension approaches $1$ from below, proving Theorem~\ref{thm:epsilon-dim}.
First, we give a hypothesis under which a finite subset $\mathcal{W}\subseteq\SL(2, \ZZ)^{\geq 0}$ gives rise to a limit set of Hausdorff dimension strictly less than one. This becomes a consequence of the fact that the limit set is a generalized Cantor set, which follows from the iterated function system formalism above.

\begin{lemma}
\label{lemma:lessthanone}
    Let $\mathcal{W}\subseteq\SL(2, \ZZ)^{\geq 0}$ be a finite set of non-identity matrices, and assume that the sets $T_{\gamma}([0, 1])$ with $\gamma\in\mathcal{W}$ are pairwise disjoint. Then $\delta(\Gamma_\mathcal{W}) <1$.
\end{lemma}

\begin{proof}
Let $\mathcal{W} = \{ \gamma_1, \ldots, \gamma_n \}$.
We will show that the limit set of the iterated function system $\{ T_\gamma : \gamma \in \mathcal{W} \}$ 
is contained in a generalized Cantor set, as defined in Beardon \cite{BeardonCantor}.  If so, Beardon guarantees the Hausdorff dimension of the limit set is less than that of the interval, i.e. less than one \cite[Theorem 8]{BeardonCantor}.

With reference to \cite[Definition 1]{BeardonCantor} and \eqref{eqn:IFSlimitset}, we define $\Delta_{i_0, \cdots, i_k} := T_{\gamma_{i_0} \cdots \gamma_{i_k}}([0,1])$, so
\[
\Lambda(\Gamma_\mathcal{W}^R) = \bigcap_{k=1}^\infty \bigcup_{i_1,\ldots,i_k=1}^{n} \Delta_{i_1, \cdots, i_k}.
\]
The initial intervals are disjoint by assumption, and each interval $\Delta_{i_0,\ldots,i_j}$ contains further disjoint subintervals $\Delta_{i_0,\ldots,i_j,i_{j+1}}$.
By the bounded distortion property and assumption, we have two consequences concerning the lengths $|\Delta_{i_0,\ldots,i_k}|$ of the intervals $\Delta_{i_0,\ldots,i_k}$:

\begin{enumerate}
    \item There is an absolute constant $1 > A > 0$ such that
\[
|\Delta_{i_1,\ldots,i_{k+1}}| \ge A |\Delta_{i_1,\ldots,i_k}|.
\]
\item There is an absolute constant $1 > B > 0$ such that for $s \neq t$, 
\[
\min\{ |x-y| : x \in \Delta_{i_1,\ldots,i_k,s}, y \in \Delta_{i_,\ldots,i_k,t} \} \ge B |\Delta_{i_1,\ldots,i_k}|.
\]
\end{enumerate}
This implies that we have constructed a generalized Cantor set.
\end{proof}

If we restrict our generators to $\Gamma_1(4)^{\geq 0}$, then the assumption of Lemma~\ref{lemma:lessthanone} can be phrased in terms of initial subwords.

\begin{lemma}\label{lem:subwordsdistintimages}
    Let $\gamma_1,\gamma_2\in\Gamma_1(4)^{\geq 0}$. Then $T_{\gamma_1}([0, 1])\cap T_{\gamma_2}([0, 1])\neq\emptyset$ if and only if either $\gamma_1$ is an initial subword of $\gamma_2$ or vice versa.
\end{lemma}
\begin{proof}
    It is clear that $T_{\gamma_1}([0, 1])$ and $T_{\gamma_2}([0, 1])$ have an interior point in common if and only if one of $\gamma_1$ and $\gamma_2$ is an initial subword of the other. 

This leaves only the case that $\gamma_1(i)=\gamma_2(j)$ for $i,j\in\{0, 1\}$. Let $\gamma_1=\genmtx$ and $\gamma_2=\sm{a'}{b'}{c'}{d'}$, and first assume $i=j=0$. Thus $\frac{b}{b+d}=\frac{b'}{b'+d'}$, and as these are reduced fractions, $b=b'$ and $d=d'$. Thus $\gamma_1=\gamma_2 R^k$ for some $k\in\ZZ$.
    The case of $i=j=1$ is similar:  $a=a'$ and $c=c'$, so $\gamma_1=\gamma_2 L^k$.

    Finally, assume $i=0$ and $j=1$, so that $\frac{b}{b+d}=\frac{a'}{a'+c'}$. Again, as these are reduced, $b=a'$ and so $d=c'$. Since our matrices are in $\Gamma_1(4)$, this gives $1\equiv d\equiv c'\equiv 0\pmod{4}$, a contradiction. 
\end{proof}

\begin{proposition}
\label{prop:epsilon-dim}
    $\Psi$ has finitely-generated subgroups $\Gamma'_n$, $n\geq 4$, of Hausdorff dimensions $\delta_n < 1$ approaching $1$.
\end{proposition}

\begin{proof}
Let 
\[
S_n = \{\genmtx\in\Psi:a,b,c,d\leq n\},\quad \Gamma'_n = \left\langle S_n\right\rangle^+.
\]
If $\gamma_1,\gamma_2\in\Psi$ are such that $\gamma_1$ is an initial subword of $\gamma_2$, then we can write $\gamma_2=\gamma_1\gamma_3$ with $\gamma_3\in\Gamma_1(4)^{\geq 0}$. By \eqref{eqn:kronchange}, it follows that $\gamma_3\in\Psi$. Furthermore, if $\gamma_1,\gamma_2\in S_n$, then $\gamma_3\in S_n$ as well. 

Therefore removing $\gamma_2$ and repeating as needed, we can reduce $S_n$ to a subset of non-identity matrices $\mathcal{W}_n$ such that $\Gamma'_n=\Gamma_{\mathcal{W}_n}$,  and no element of $\mathcal{W}_n$ is an initial subword of another element.

By Lemma~\ref{lem:subwordsdistintimages}, the assumptions in Lemma~\ref{lemma:lessthanone} hold, so  $\delta(\Gamma'_n)=\delta(\Gamma_{\mathcal{W}_n})<1$. 

Note that for any finite $X \subseteq \Psi$, $\Gamma_X \subseteq \Gamma_n'$ for some $n$.  
Furthermore, the collection $\{T_{\gamma}:\gamma\in\mathcal{W}_n\}$ forms an infinite conformal iterated function system, and by Theorem 3.15 of \cite{MU}, 
\[
1 = \delta(\Psi) = \sup_{X \subseteq \Psi, |X| < \infty} \delta(\Gamma_X) = \lim_{n\rightarrow\infty} \delta(\Gamma'_n),
\]
giving the result.
\end{proof}

As a corollary to the proof, we observe that every finitely-generated subroup of $\Psi$ has Hausdorff dimension less than $1$.  Since $\Psi$ has Hausdorff dimension $1$, it is not finitely generated.  This proves Proposition~\ref{prop:psidim1}.  As a matter of interest, this can be proven constructively.

\begin{lemma}
\label{lemma:psi-infinite}
    Let
    \[
M_k := \sm{12k+1}{4k}{12k+4}{4k+1} \in \Psi, \quad k \ge 0.
\]
Then $M_k$ has no non-trivial initial subword in $\Gamma_1(4)^{\ge 0}$. 
In particular, $\Psi$ is not finitely generated.   
\end{lemma}

\begin{proof}
Recall that $\SL(2, \ZZ)^{\ge 0}$ is generated by $L$ and $R$, with no relations.
    Observe that, in $\SL(2, \ZZ)$ modulo $4$, multiplication by $L$ on the left induces a cycle of order $4$ on $R^3 = \sm{1}{0}{3}{1}$:
    \[
    \lm{1}{0}{3}{1} \rightarrow
    \lm{0}{1}{3}{1} \rightarrow
    \lm{3}{2}{3}{1} \rightarrow
    \lm{2}{3}{3}{1} \rightarrow
    \lm{1}{0}{3}{1}
    \]
    Hence $L^nR^3 \notin \Gamma_1(4)$ for any $n \geq 0$.
    Thus any word of the form
    \[
    M_k = RL^{4k}R^3=\lm{12k+1}{4k}{12k+4}{4k+1}, \quad k \ge 0,
    \]
    is not a product of two non-trivial subwords lying in $\Gamma_1(4)^{\ge 0}$. As $M_k\in \Psi \subseteq \Gamma_1(4)^{\ge 0}$, there can be no finite set of generators for $\Psi$.
\end{proof}

The proof of Theorem~\ref{thm:epsilon-dim} depends on Proposition~\ref{prop:psi1table}, proven in the next section.  In particular, to show that an orbit avoiding squares is actually an example of a reciprocity obstruction, we need to control the congruence obstructions, so that we know they are not themselves ruling out squares.  An easy way to do this is to begin with a finitely generated subsemigroup whose congruence obstructions are known.  We define and investigate a few in the next section.

\begin{proof}[Proof of Theorem~\ref{thm:epsilon-dim}]
Note that $\Psi_1\subseteq\Gamma'_n\subseteq\Psi$ for all $n\geq 4$ ($\Psi_1$ is studied in the next section). By Proposition~\ref{prop:psi1table}, the congruence obstructions for $\Psi_1$ and $\Psi$ are identical, hence they hold for $\Gamma'_n$ as well. In particular, any orbit $\Psi\smcol{x}{y}$ that exhibits reciprocity obstructions ensures that $\Gamma'_n\smcol{x}{y}$ also exhibits reciprocity obstructions.
\end{proof}

\section{Obstructions in $\Psi_1$ and $\Psi_2$}
\label{sec:psi12}

While $\Psi$ is not finitely generated, we can consider finitely generated subsemigroups. Recall the following definitions from Section~\ref{sec:results}:
\[
\Psi_1 = \left\langle L, R^4 \right\rangle^+ = \left\langle\lm{1}{1}{0}{1}, \lm{1}{0}{4}{1}\right\rangle^+,\qquad
\Psi_2 = \left\langle L^4, R^4 \right\rangle^+ =\left\langle\lm{1}{4}{0}{1}, \lm{1}{0}{4}{1}\right\rangle^+.
\]

As subsemigroups of $\Psi$, their orbits avoid squares.  However, this fact is most interesting when squares are allowed by the congruence obstructions (turning them into reciprocity obstructions), and the Hausdorff dimension is sufficiently large.

For $\Psi_2$, by Corollary \ref{cor:psi2and0mod4} (proven later this section), it suffices to compute the Hausdorff dimension for continued fractions built from the alphabet $4\ZZ^+$. For this estimate, we thank Mark Pollicott for sharing with us his code to rigorously compute the Hausdorff dimensions for alphabets built from arithmetic progressions; see \cite{PollicottVytnova}. His code gives $\delta_2=0.6045578\pm 10^{-7}$.

For $\Psi_1$, the Hausdorff dimension certainly meets or exceeds that of $\Psi_2$.  For a simple non-rigorous estimate, we assume a result in the nature of Hensley's growth estimate \eqref{eqn:hensley} holds (from Section 1).  We compute the orbit of $\Psi_1\smcol{1}{1}$ up to $50000$, and estimate the growth rate using a least squares regression. This is done in the function \texttt{test\_psi1\_hdim} in \cite{GHSemigroup}, and results in the estimate of $\delta_1\approx 0.719$.

We classify the possible congruence obstructions in the next section.

\subsection{Strong approximation}
\label{sec:strongapprox}
We claim that any congruence obstruction in an orbit of $\Psi_1$ occurs modulo 4, and of $\Psi_2$ occurs modulo 16.  In other words, we show an explicit strong approximation property with the ``bad modulus'' being $2^2$ or $2^4$, respectively.

For a subgroup $G$ of $\SL(2, \ZZ)$ and a positive integer $n$, denote by $G/n$ the image of the reduction map $\rho_n:G\rightarrow\SL(2, \ZZ/n\ZZ)$. The full statement of explicit strong approximation is as follows.

\begin{theorem}\label{thm:strongapprox}
Let $i \in \{1,2\}$. Writing $n = \opart{n} n'$, with $\opart{n}$ odd and $n'$ a power of $2$, we have:
\begin{enumerate}
\item For $i=1,2$, $\Psi_i/n \cong \Psi_i/\opart{n} \times \Psi_i/n'$.
\item $\Psi_i/\opart{n} = \SL(2,\ZZ)/\opart{n}$.
\item There exists an exponent $e_i \ge 1$ such that for all $k \ge e_i$,
\[
(\Psi_i/2^k) / (\Psi_i/2^{e_i}) = (\SL(2,\ZZ)/2^k) / (\SL(2,\ZZ)/2^{e_i}).
\]
\item We have $e_1 = 2$ and $e_2 = 4$.
\end{enumerate}
\end{theorem}

The group $\SL(2, \ZZ)$ has a multiplicative structure, as a consequence of strong approximation for $\SL(2,\ZZ)$:
\[
\SL(2,\ZZ)/mn \cong \SL(2,\ZZ)/m \times \SL(2,\ZZ)/n, \quad \gcd(n,m)=1.
\]

\begin{lemma}\label{lem:coprimecombine}
	 Let $\gcd(m,n)=1$, and let $i \in \{1,2\}$. Then $\Psi_i/{mn}\cong \Psi_i/{m}\times \Psi_i/n$.
\end{lemma}
\begin{proof}
  We begin with $\Psi_2$.  Using the multiplicative structure of $\SL(2,\ZZ)$, suppose $M \in \SL(2,\ZZ)/mn$ maps to $(M_1, M_2) \in \Psi_2/m \times \Psi_2/n$.  We wish to show $M \in \Psi_2/mn$.
	There exists words $W_1$ and $W_2$ in $L^4$ and $R^4$ which descend to $M_1$ modulo $m$ and $M_2$ modulo $n$, respectively. Note that the order of  $L^4$ (respectively $R^4$) modulo $2^e N$, for $N$ odd, is $N \max(1, 2^{e-2})$. The orders of $L^4$ (respectively, $R^4$) modulo $m$ and $n$ are therefore coprime.  Thus, by the Chinese remainder theorem we can modify the exponents in $W_1$ and $W_2$ to produce a single word $W$ in $L^4$ and $R^4$ that reduces to $M_1$ and $M_2$ modulo $m$ and $n$ respectively.  It is clear that the $\Psi_1$ case works mutatis mutandis.
\end{proof}
 
\begin{lemma}\label{lem:oddsurj}
	Let $n$ be an odd positive integer.  Then $\Psi_2/n = \Psi_1/n = \SL(2,\ZZ)/n$.
\end{lemma}
\begin{proof}
As $n$ is odd, $(L^4)^k \equiv L \pmod{n}$ for some $k$, hence $\rho_n(L)\in \Psi_2/n$; similarly, $\rho_n(R)\in \Psi_2/n$.  Therefore $\SL(2, \ZZ)^{\ge 0}/n \subseteq \Psi_2/n$, including $\sm{n}{1}{n^2-1}{n} \equiv \sm{0}{1}{-1}{0} \pmod{n}$. But $S := \sm{0}{1}{-1}{0}$ together with $L$ and $R$ generate $\SL(2, \ZZ)$, giving $\Psi_2/n = \SL(2,\ZZ)/n$. The final conclusion is a result of the inclusions $\Psi_2 \subseteq \Psi_1 \subseteq  \SL(2, \ZZ)$.
\end{proof}

\begin{lemma}\label{lem:16enough}
	Let $e\geq 5$ be a positive integer, and let $M\in \SL(2,\ZZ)/2^e$. Then $M\in \Psi_2/{2^e}$ if and only if $\rho_{16}(M)\in \Psi_2/{16}$.
\end{lemma}
\begin{proof}
We have a group
\[
G = \left\{ \genmtx : b \equiv c \equiv 0 \bmod{4},\;\; a \equiv d \equiv 1 \bmod{16} \right\} \subseteq \SL(2, \ZZ).
\]
In particular, $L^4, R^4 \in G$, and in fact $G/16 = \Psi_2/16$.

Now, assume $M= \genmtx \in \SL(2,\ZZ)/2^e$ satisfies $\rho_{16}(M) \in G/16$. By possibly replacing $M$ by $R^{-4}M=\sm{a}{b}{c-4a}{d-4b}$, we can assume that $v_2(c)=2$.  Observe that $v_2(4kc)=4+v_2(k)$, so there exists a $k$ such that $a+4kc\equiv 1\pmod{2^e}$.  Thus we replace $M$ with $L^{4k}M=\sm{a+4kc}{b+4kd}{c}{d}$, so we can assume $a\equiv 1 \pmod{2^e}$. Next, replacing $M$ with $R^{4\ell}M$ for an appropriate $\ell$, we can assume $M = \sm{1}{b}{0}{d}\pmod{2^e}$. As this has determinant $1$, we observe that $d \equiv 1\pmod{2^e}$, and in fact $M \equiv L^{b}\pmod{2^e}$. As $\rho_{16}(M) \in G/2^e$, $b$ is a multiple of 4, so $b=4b'$. Going back to our original element $M$, we have written
\[M = R^{\text{0 or 4}}L^{-4k}R^{-4\ell}L^{4b'},\]
so by choosing $-k,-\ell,b'\geq 0$, we have written $M$ as a word in $L^4$ and $R^4$, proving that $M\in \Psi_2/{2^e}$.
\end{proof}

\begin{lemma}\label{lem:4enough}
	Let $e\geq 3$ be a positive integer, and let $M\in \SL(2,\ZZ)/2^e$. Then $M\in \Psi_1/{2^e}$ if and only if $\rho_4(M)\in \Psi_1/{4}$.
\end{lemma}
\begin{proof}
We have $L, R^4 \in \Gamma_1(4)$, and in fact $\Gamma_1(4)/4 = \Psi_1/4$.  The proof is almost exactly as in the last lemma:  assume $M:= \genmtx  \in \SL(2,\ZZ)/2^e$ satisfies $\rho_{4}(M) \in \Gamma_1(4)/4$, and write it as a word in $L$ and $R^4$.
\end{proof}

\begin{proof}[Proof of Theorem~\ref{thm:strongapprox}]
	Lemmas \ref{lem:coprimecombine}, \ref{lem:oddsurj}, and \ref{lem:16enough} combine to give the congruence obstructions for $\Psi_2$, while Lemmas \ref{lem:coprimecombine}, \ref{lem:oddsurj}, and \ref{lem:4enough} give the result for $\Psi_1$. 
\end{proof}

\subsection{Reciprocity obstructions for $\Psi_1$ and $\Psi_2$}
\label{sec:twoexamples}
We will now prove the next set of results from Section \ref{sec:results}. Proposition~\ref{prop:twoexamples} gives orbits of $\Psi_1$ and $\Psi_2$ which exhibit reciprocity obstructions.

\begin{proof}[Proof of Proposition~\ref{prop:twoexamples}]
    By Theorem~\ref{thm:strongapprox}, the congruence obstructions in $\Psi_1 \smcol{2}{3}$ occur modulo 4. Using $\Psi_1/4=\Gamma_1(4)/4$, it follows that the numerator has no congruence obstructions, and the missing squares result in a reciprocity obstruction.

    For $\Psi_2\smcol{3}{8}$, the congruence obstructions occur modulo 16. From the proof of Lemma \ref{lem:16enough}, it suffices to analyze the denominators of $\sm{1\pmod{16}}{0\pmod{4}}{0\pmod{4}}{1\pmod{16}}\smcol{3}{8}$, which comes out to be $0\pmod{4}$.
\end{proof}

Proposition~\ref{prop:psi1table} states that Table~\ref{table:psi1obstructions} still holds for orbits of $\Psi_1$.

\begin{proof}[Proof of Proposition~\ref{prop:psi1table}]
    By Theorem~\ref{thm:strongapprox}, the congruence obstructions in an orbit $\Psi_1 \smcol{x}{y}$ are controlled by $(x, y)\pmod{4}$, and an easy computation shows they are exactly as described in Table~\ref{table:psi1obstructions}. In particular, the reciprocity obstructions predicted remain reciprocity obstructions, and the table remains unchanged.
\end{proof}

If Table~\ref{table:psi1obstructions} does not predict a reciprocity obstruction (when squares are possible based on the congruence obstructions), then we do not have a proof that the orbit $\Psi_1\smcol{x}{y}$ \emph{does} represent squares. In particular, we do not know if Table~\ref{table:psi1obstructions} is \emph{complete} for $\Psi_1$. The best we can do is compute a large number of orbits, and verify that it appears to be complete.

In the numerators/denominators of $\Psi_1\smcol{x}{y}$ with $1\leq x, y\leq 200$, where a square was not ruled out by congruence obstructions and a reciprocity obstruction was not predicted by Table~\ref{table:psi1obstructions}, we found a square numerator/denominator below $400^2$ in all cases. Code to reproduce this output is found in \cite{GHSemigroup} with the function \texttt{test\_table1\_psi1\_many}. This data supports Conjecture~\ref{conj:psi1localglobal}, which predicts that Table~\ref{table:psi1obstructions} is complete for $\Psi_1$.

Finally, consider the orbit $\Psi_1\smcol{2}{3}$. By computing the numerators up to $10^7$, we observe that the final missing non-square is 10569, lending support to Conjecture~\ref{conj:twoexamples}. This computation can be recreated with \texttt{test\_psi1\_23orbit}.

\section{Connection to continued fractions}
\label{sec:ctd}

The final order of business is to prove Proposition~\ref{prop:psi12incontfrac} and the subsequent results, which connect the orbits of $\Psi_1,\Psi_2$ to continued fractions.

\begin{proof}[Proof of Proposition~\ref{prop:psi12incontfrac}]

Assume $\smcol{u}{v}\in\Psi_2\smcol{x}{y}$, and write
\[\smcol{u}{v}=L^{4b_0}R^{4b_1}L^{4b_2}R^{4b_3}\cdots L^{4b_{2m}}R^{4b_{2m+1}}\smcol{x}{y}=:W_1\smcol{x}{y},\]
for nonnegative integers $b_i$, where $b_i>0$ if $0<i<2m+1$. From the discussion at the start of Section~\ref{sec:results}, we have
\[\smcol{x}{y}=R^{a_1}L^{a_2}R^{a_3}\cdots R^{a_{2n-1}}\smcol{1}{0}=:W_2\smcol{1}{0},\]
Thus $\smcol{u}{v}=W_1W_2\smcol{1}{0}$, and the even continued fraction of $\frac{u}{v}$ is the concatenation of $W_1$ with $W_2$.

It is clear that these steps are reversible, i.e. it is an if and only if. The case of $\Psi_1\smcol{x}{y}$ follows in an analogous fashion.
\end{proof}

Relying on the even continued fraction is a mildly annoying technicality caused by working with the words in $L,R$. This can be removed by considering inverses of rational numbers.

\begin{proposition}\label{prop:psi2nomoreevencontfrac}
    Let $x<y$ be coprime nonnegative integers, and let $\frac{x}{y}=[0;a_1,a_2,\ldots,a_n]$ be a continued fraction expansion. For coprime nonnegative integers $u,v$, we have either $\smcol{u}{v}\in\Psi_2\smcol{x}{y}$ or $\smcol{v}{u}\in\Psi_2\smcol{x}{y}$ if and only if $\frac{u}{v}$ has a continued fraction of the form
    \[
    \frac{u}{v}=
        [4b_0;4b_1, 4b_2, \ldots, 4b_{m-1}, 4b_m+ a_1, a_2, \ldots, a_n],
    \]
    where the $b_i$ are nonnegative integers with $b_i>0$ if $1\leq i\leq m-1$.
\end{proposition}
\begin{proof}
    If $x=0$, we must have $\frac{0}{1}=[0]$, and the condition of either $\smcol{u}{v}$ or $\smcol{v}{u}$ is in $\Psi_2\smcol{0}{1}$ is equivalent to either $\smcol{u}{v}$ or $\smcol{v}{u}$ is in $\Psi_2\smcol{1}{0}$. The condition that $\frac{u}{v}$ has a continued fraction involving only multiples of 4 is equivalent to either $\frac{u}{v}$ or $\frac{v}{u}$ has an even continued fraction involving only multiples of 4, and Proposition~\ref{prop:psi12incontfrac} finishes this case.
    
    Now assume $x>0$, and note that if $\frac{u}{v}$ takes the above form, then it will still take this form if we use the other continued fraction expansions of $\frac{u}{v}$ and $\frac{x}{y}$. Therefore we can assume that $n$ is odd, i.e. we took the even continued fraction expansion of $\frac{x}{y}$.
    
    If $c_0>0$, note that
    \[\frac{1}{[c_0;c_1,c_2,\ldots,c_k]}=[0;c_0,c_1,\ldots,c_k].\]
    In particular, if $\frac{u}{v}$ takes the assumed form, possibly replacing $\frac{u}{v}$ by $\frac{v}{u}$ allows us to assume the total length is even, and therefore we fall under the purview of Proposition~\ref{prop:psi12incontfrac}, proving the result.
\end{proof}

Corollary~\ref{cor:psi2and0mod4} follows immediately from Proposition~\ref{prop:psi2nomoreevencontfrac}.

Theorem~\ref{thm:zarembareciprocity} follows from checking the congruence obstructions for $\Psi_{\mathcal{A}}\smcol{x}{y}$ in the relevant cases, as follows.

\begin{lemma}\label{lem:zarembareicprocitydenom}
    Let $\mathcal{A}$ be a subset of $4\ZZ^+$ containing $4,8$, and let $x,y$ be odd. Then the congruence obstruction on a denominator $d$ of $\Psi_{\mathcal{A}}\smcol{x}{y}$ is $d\equiv x\text{ or }y\pmod{4}$.
\end{lemma}
\begin{proof}
Write $C_a = \sm{0}{1}{1}{a}$, and $U = \sm{0}{1}{1}{0}$. We have that
    \[
    \Psi_\mathcal{A}^2 := \langle C_aC_b :a,b \in \mathcal{A} \rangle^+ = \langle UL^aR^bU : a,b \in \mathcal{A} \rangle^+ \subseteq U\Psi_2U = \Psi_2 \subseteq \SL(2,\ZZ).
    \]
 \textbf{Claim:} For any modulus $m$, $\Psi_2/m = \Psi_\mathcal{A}^2/m$.

    Indeed, let $m$ be a modulus.  Within $\SL(2,\ZZ/m\ZZ)$, the semigroup generated by a list of generators is the same as the group generated by those generators.  Therefore, $\Psi_\mathcal{A}^2/m$ contains
    \[
    C_8C_4 (C_4C_4)^{-1} = R^4, \quad
    (C_4C_4)^{-1} C_4C_8 = L^4.
    \]
    Hence $\Psi_2/m \subseteq \Psi_\mathcal{A}^2/m 
    \subseteq \Psi_2/m$, which proves the claim.

As a consequence, we know the congruence obstructions for $\Psi_\mathcal{A}^2$ are those of $\Psi_2$. A denominator of $\Psi_{\mathcal{A}}\smcol{x}{y}$ is a denominator of either $\Psi_{\mathcal{A}}^2\smcol{x}{y}$ or $\Psi_{\mathcal{A}}^2C_a\smcol{x}{y}$ for some $a\in\mathcal{A}$. Congruence obstructions for $\Psi_2$ occur modulo 16, and the proof of Lemma~\ref{lem:16enough} gives that 
\[\Psi_2/16=\left\{\lm{1}{4b}{4c}{1}\pmod{16}:0\leq b,c\leq 3\right\}.\]
Thus the denominators in $\Psi_{\mathcal{A}}^2\smcol{x}{y}$ are $y\pmod{4}$, and the denominators in $\Psi_{\mathcal{A}}^2C_a\smcol{x}{y}$ are $x\pmod{4}$, finishing the proof.    
\end{proof}

\begin{proof}[Proof of Theorem~\ref{thm:zarembareciprocity}]
By Lemma~\ref{lem:zarembareicprocitydenom}, the denominators of $\Psi_{\mathcal{A}}\smcol{x}{y}$ have congruence obstruction either $1\pmod{4}$ or $1\pmod{2}$, both of which admit squares. Thus, it suffices to show that the denominator cannot be a square.

Next, $\smcol{u}{v}\in\Psi_{\mathcal{A}}\smcol{x}{y}$ if and only if we can write the continued fraction for $\frac{u}{v}$ as
\[[0;b_1, b_2, \ldots, b_{n-1}, b_n+a_0, a_1, \ldots, a_n],\]
where $\frac{x}{y}=[a_0;a_1, a_2, \ldots, a_n]$ and $b_i\in\mathcal{A}$ for all $i$.

By Proposition~\ref{prop:psi2nomoreevencontfrac}, this implies that either $\smcol{u}{v}$ or $\smcol{v}{u}$ is in the orbit of $\Psi_2\smcol{x}{y}$ or $\Psi_2\smcol{y}{x}$ (the second case occurs when $a_0>0$, and one considers $\frac{y}{x}=[0;a_0,a_1,\ldots,a_n]$ instead). Due to the symmetry of $\Psi_2$, we see that if $v$ is a denominator in $\Psi_{\mathcal{A}}\smcol{x}{y}$, then it is either a numerator or denominator in $\Psi_2\smcol{x}{y}$, hence $\Psi\smcol{x}{y}$ as well. By Theorem \ref{thm:semigroupreciprocity} and the assumption that $\kron{x}{y}=-1$, $v$ is not a square, as claimed.
\end{proof}

Specializing Theorem~\ref{thm:zarembareciprocity} to $\Psi_{\mathcal{A}}\smcol{3}{5}$ gives Corollary~\ref{cor:psi2threefiveorbitcontfracwhydoesjamesusesuchlongnames}.  For the second equivalent phrasing, we need only observe that
\[L(\Psi_{\mathcal{A}} \smcol{0}{1}) = (2 \; 3) \Psi_{\mathcal{A}} \smcol{0}{1} = (0 \; 1) \Psi_{\mathcal{A}} \smcol{2}{3},\]
since $\Psi_{\mathcal{A}}$ is invariant under transpose. Finally, let's consider the more general Corollary~\ref{cor:bigcontfracnosquares}, which comes about by taking the orbits of $\smcol{3}{2+3k}$ for all $k\geq 0$.

\begin{proof}[Proof of Corollary~\ref{cor:bigcontfracnosquares}]
    We have $\frac{u}{v}=[0;4a_1, 4a_2, \ldots, 4a_n, k, 1, 2]$ if and only if $\smcol{u}{v}\in\Psi_{4\ZZ^+}\smcol{3}{2+3k}$. Following the final analysis of the proof of Lemma~\ref{lem:zarembareicprocitydenom}, we see that there are no congruence obstructions on the denominator. From the proof of Theorem~\ref{thm:zarembareciprocity}, it suffices to show that the orbits $\Psi\smcol{3}{2+3k}$ never admit squares in the numerator or denominator. This follows from an analysis of Table~\ref{table:psi1obstructions} with the four values of $k\pmod{4}$. The numerator is always $3\pmod{4}$, and thus not a square by congruence. The denominator is not a square if $k= 0, 3\pmod{4}$ by a congruence obstruction, and it is not a square if $k= 1, 2\pmod{4}$ by a reciprocity obstruction.
\end{proof}

All continued fractions of the form $[0;4a_1, 4a_2, \ldots, 4a_n, a_{n+1}, 1, 2]$ with denominator bounded by $B$ can be enumerated by a depth first search. This is implemented in the function \texttt{contfrac\_tail\_missing}, and parallelized in the method \texttt{cfracsearch} in \cite{GHSemigroup}. A lengthy computation up to $B=2\times 10^{13}$ (mostly done on 64 cores, for a rough total of 150000 CPU hours) yields fewer and fewer missing denominators, finally stopping around $7.97\times 10^{12}$. The final missing denominators found in each equivalence class modulo 4 are listed in Table~\ref{table:contfraclast}.

\begin{table}[hbt]
\centering
\caption{Non-square $d$ which are not denominators of numbers of the form $[0;4a_1, 4a_2, \ldots, 4a_n, a_{n+1}, 1, 2]$ up to $B=2\times 10^{13}$.}\label{table:contfraclast}
\begin{tabular}{|c|r|r|r|} 
\hline
$d\pmod{4}$ & Total missing & Largest missing & $\approx B/\text{Largest}$\\ \hline
0 & 2869016 & 2281020902160 & 8.77 \\ \hline
1 & 1028567 & 674124164325 & 29.67 \\ \hline
2 & 6654884 & 7968219670470 & 2.51 \\ \hline
3 & 9665 & 2308640415 & 8663.11\\ \hline
\end{tabular}
\end{table}

In particular, Conjecture \ref{conj:contfraceventuallyall} seems plausible.

\section{Subgroups of $\SL(2, \ZZ)$ with reciprocity obstructions}\label{section:groups}

We end with a few remarks about the case of groups (instead of semigroups).  In Section 4 of \cite{ApolloniusZaremba}, Kontorovich studies a thin orbit arising from Pythagorean triples, in particular from the group
\[\tilde{\Gamma}:=\left\langle\pm\lm{1}{2}{0}{1}, \pm\lm{1}{0}{4}{1}\right\rangle = \left\langle\pm L^2, \pm R^4\right\rangle.\]
This group sits between the groups generated by $\Psi_1$ and $\Psi_2$.  The denominators of $\Psi_2\smcol{3}{8}$ are restricted to be $0\pmod{4}$ and not square.  However,  $\tilde{\Gamma}\smcol{3}{8}$ can have square denominators:
\[R^4L^{-2}R^4L^{-2}\lmcol{3}{8}=\lmcol{75}{256}.\]

This discussion raises a question: is our restriction to semigroups essential?

The failure of our approach for groups is the inevitable involvement of the place at infinity.  Roughly speaking, quadratic reciprocity says, for non-negative numbers, that the behaviour of $\kron{\cdot}{x}$ is the same as the behaviour of $\kron{x}{\cdot}$.  Our semigroup orbits involve only non-negative integers.  However, $\kron{x}{-1}$ tracks the sign of $x$, and $\kron{-1}{x}$ tracks $\opart{x}\pmod{4}$.
The example above is instructive: $R^4L^{-2}\smcol{3}{8}=\smcol{-13}{-44}$, so both numbers have dipped into the negatives. At this point, we still have $\kron{3}{8}=-1=\kron{-13}{-44}$, but after one application of $L^{-2}$ we get $\smcol{75}{-44}$, where now $\kron{75}{-44}=1$.

\begin{question}
    Do there exist subgroups of $\SL(2, \ZZ)$ that exhibit reciprocity obstructions?
\end{question}

A response to this question would require a precise formulation of the notion of reciprocity obstruction.

\bibliographystyle{alpha}
\bibliography{refs}
\end{document}